\documentclass[a4paper,11pt]{amsart}

\usepackage[noBBpl]{mathpazo}
\usepackage{genyoungtabtikz}
\usepackage{xspace,amsmath,amssymb,nccmath,amsthm,needspace,enumitem,pdflscape,stmaryrd,rotating,mathtools}
\usepackage[hidelinks,bookmarks=false,pagebackref,bookmarks=true]{hyperref}
\usepackage[capitalise]{cleveref}
\usepackage[margin=2cm,includehead,a4paper]{geometry}
\usepackage{mathrsfs}

\numberwithin{equation}{section}

\newtheoremstyle{mattthm}{}{}{\slshape}{}{\bfseries}{.}{ }{}
\theoremstyle{mattthm}

\newtheorem{thm}[equation]{Theorem}
\newtheorem{lemma}[equation]{Lemma}
\newtheorem{cory}[equation]{Corollary}
\newtheorem{propn}[equation]{Proposition}

\newtheoremstyle{mattdef}{}{}{}{}{\bfseries}{.}{ }{}

\theoremstyle{mattdef}

\newtheorem*{wng}{Warning}
\newtheorem*{acks}{Acknowledgements}

\makeatletter
\def\thmhead@plain#1#2#3{%
  \thmname{#1}\thmnumber{\@ifnotempty{#1}{ }\@upn{#2}}%
  \thmnote{ {\the\thm@notefont#3}}}
\let\thmhead\thmhead@plain
\makeatother

\crefname{thm}{Theorem}{Theorems}
\crefname{theor}{Theorem}{Theorems}
\crefname{propn}{Proposition}{Propositions}
\crefname{lemma}{Lemma}{Lemmas}
\crefname{cor}{Corollary}{Corollaries}
\crefname{conj}{Conjecture}{Conjectures}
\Crefname{thm}{Theorem}{Theorems}
\Crefname{theor}{Theorem}{Theorems}
\Crefname{propn}{Proposition}{Propositions}
\Crefname{lemma}{Lemma}{Lemmas}
\Crefname{cor}{Corollary}{Corollaries}
\Crefname{conj}{Conjecture}{Conjectures}

\setitemize[1]{label={$\diamond$}}

\newcommand\con{cited~on~p.~}
\newcommand\cons{cited~on~pp.~}

\renewcommand*{\backref}[1]{}
\renewcommand*{\backrefalt}[4]{\ifcase #1 \hspace*{\fill}{\small [no~citations.]}\or\hspace*{\fill}{\small [\con#2]}\else\hspace*{\fill}{\small [\cons#2]}\fi}
\newcommand\bkp{\backrefprint\renewcommand\con{}\renewcommand\cons{}}

\makeatletter
\renewenvironment{proof}[1][\proofname] {\par\pushQED{\qed}\normalfont\topsep6\p@\@plus6\p@\relax\trivlist\item[\hskip\labelsep\bfseries#1\@addpunct{.}]\ignorespaces}{\popQED\endtrivlist\@endpefalse}
\makeatother

\begin{document}

\newcommand\qalg[2]{(#1\mid#2)}
\newcommand\ii{\mathrm i}
\newcommand\jj{\mathrm j}
\newcommand\kk{\mathrm k}

\newcommand\str{(\textasteriskcentered)}
\newcommand\imps[2]{\item[\rm(\ref{#1})$\Rightarrow$(\ref{#2})]}
\renewcommand\leq\leqslant
\renewcommand\geq\geqslant
\newcommand\pmx[4]{\begin{pmatrix}#1&#2\\#3&#4\end{pmatrix}}
\newcommand\spmx[4]{\text{\small$\begin{pmatrix}#1&#2\\#3&#4\end{pmatrix}$}}
\renewcommand\le\leqslant
\renewcommand\ge\geqslant
\newcommand{\dom}{\trianglerighteqslant}
\newcommand{\doms}{\vartriangleright}
\newcommand{\ndom}{\ntrianglerighteqslant}
\newcommand{\ndoms}{\not\vartriangleright}
\newcommand{\domby}{\trianglelefteqslant}
\newcommand{\domsby}{\vartriangleleft}
\newcommand{\ndomby}{\ntrianglelefteqslant}
\newcommand{\ndomsby}{\not\vartriangleleft}
\newcommand\nequiv{\not\equiv}
\newcommand\call{\mathcal L}
\newcommand\lset[2]{\left\{\left.#1\ \right|\ \smash{#2}\right\}}
\newcommand\rset[2]{\left\{\smash{#1}\ \left|\ #2\right.\right\}}
\newcommand\sm\setminus
\newcommand\ls\leqslant
\newcommand\gs\geqslant
\newcommand\bbc{\mathbb C}
\newcommand\C{\bbc}
\newcommand\bbq{\mathbb Q}
\newcommand\Q{\bbq}
\newcommand\bbr{\mathbb R}
\newcommand\R{\bbr}
\newcommand\bbf{\mathbb F}
\newcommand\F{\bbf}
\newcommand\bbn{\mathbb N}
\newcommand\bbz{\mathbb Z}
\newcommand\Z{\bbz}
\newcommand\al\alpha
\newcommand\la\lambda
\newcommand\si\sigma
\renewcommand\phi\varphi
\newcommand\eps\varepsilon
\newcommand\ep\varepsilon
\newcommand\ppmod[1]{\ (\operatorname{mod}\,#1)}
\renewcommand\iff{if and only if\xspace}
\newcommand\hsss{\hat{\mathfrak{S}}_}
\newcommand\haaa{\hat{\mathfrak{A}}_}
\newcommand\len[1]{h(#1)}
\newcommand\ch[1]{\lan#1\ran}
\newcommand\ach[1]{\lan#1\ran}
\newcommand\hf{.5}
\newcommand\aaa{\mathfrak{A}_}

\newcommand\bt[1]{{\mathbf{S}(#1)}}
\newcommand\bta[1]{{\mathbf{T}(#1)}}
\newcommand\bd[1]{\mathbf{D}(#1)}
\newcommand\be[1]{\mathbf{E}(#1)}
\newcommand\spe[1]{\operatorname{S}(#1)}
\newcommand\spepm[1]{\operatorname{S}(#1)_\pm}
\newcommand\spex[2]{\operatorname{S}^{#2}(#1)}
\newcommand\spes[1]{\spe{#1}_\ast}
\newcommand\ape[1]{\operatorname{T}(#1)}
\newcommand\apepm[1]{\operatorname{T}(#1)_\pm}
\newcommand\smp[1]{\operatorname{D}(#1)}
\newcommand\amp[1]{\operatorname{E}(#1)}
\newcommand\jms[1]{\operatorname{D}^{#1}}
\newcommand\sspe[1]{\operatorname{S}^{#1}}
\newcommand\apes[1]{\ape{#1}_\ast}
\newcommand\smps[1]{\smp{#1}_\ast}
\newcommand\amps[1]{\amp{#1}_\ast}
\newcommand\sns{\operatorname{S}}
\newcommand\snd{\operatorname{D}}
\newcommand\ant{\operatorname{T}}
\newcommand\ane{\operatorname{E}}

\newcommand\sss{\mathfrak{S}_}
\newcommand\ol[1]{\overline{#1}}
\newcommand\dn[2]{\mathrm{d}_{#1#2}}
\newcommand{\da}{{\downarrow}}
\newcommand{\ua}{{\uparrow}}
\newcommand{\End}{\operatorname{End}}
\newcommand{\ind}{\operatorname{ind}}
\newcommand{\nchar}{\operatorname{char}}
\newcommand{\sgn}{\operatorname{sgn}}
\newcommand{\triv}{\operatorname{triv}}
\newcommand{\Mat}{\mathrm{Mat}}
\newcommand{\gl}{\mathrm{GL}}
\newcommand{\Aut}{\mathrm{Aut}}
\newcommand{\reg}{{\tt R}}
\newcommand{\mull}{{\tt M}}

\newcommand\lan\langle
\newcommand\ran\rangle
\newcommand\chr[1]{\lan#1\ran}
\newcommand\bchr[1]{\varphi(#1)}

\def\Par{{\mathscr {P}}}
\def\Parinv{\mathscr {P}^{\mathfrak{A}}}

\title[Representations that remain irreducible modulo every prime]{Representations of symmetric and alternating groups and their double covers that remain irreducible modulo every prime}

\author{Matthew Fayers}
\address{Queen Mary University of London\\Mile End Road\\London E1 4NS\\U.K.}
\email{m.fayers@qmul.ac.uk}

\author{Lucia Morotti}
\address{Department of Mathematics, University of York, York, YO10 5DD, UK}
\email{lucia.morotti@york.ac.uk} 


\thanks{The second author was supported by the Royal Society grant URF$\backslash$R$\backslash$221047. While starting to work on this paper the second author was working at the Mathematisches Institut of the Heinrich-Heine-Universität D\"usseldorf.}

\maketitle

\begin{abstract}
We classify globally irreducible representations of alternating groups and double covers of symmetric and alternating groups. In order to achieve this classification we also completely characterise irreducible representations of such groups which reduce almost homogeneously in every characteristic. This also allows us to classify irreducible representations that remain irreducible in every characteristic as well as irreducible representations of these groups that can appear as composition factors of globally irreducible representations of groups containing $\aaa n$ or $\haaa n$ as normal subgroups. In particular we show that, apart from finitely many exceptions, for any of these questions such representations are either $1$-dimensional or basic spin representations.
\end{abstract}

\tableofcontents

\section{Introduction}

Globally irreducible representations of finite groups were introduced by Gross in \cite{gr}, generalising notations defined by Thompson in \cite{th}, and were studied further by Tiep in \cite{t2}. They are defined as representations over the field $\Q$ which remain irreducible when scalars are extended to $\R$, and for which certain reductions to positive characteristic $p$ remain irreducible for every prime $p$, see \cref{gir} for more details.

It is a natural and important question to classify the globally irreducible representations of a given finite group, but this has been accomplished for very few families of groups. For the symmetric groups $\sss n$, this was done by Kleshchev and Premet, who proved the following.

\begin{thm}\cite[Theorem A]{kp}\label{mainsn}
Let $M$ be a representation of $\sss n$ over $\bbq$. Then $M$ is globally irreducible \iff $M$ is $1$-dimensional.
\end{thm}

In fact, since (by \cite[Theorem 11.5]{JamesBook}) any field is a splitting field for $\sss n$, it is easy to see from the definition of globally irreducible representations that any ordinary irreducible representation of $\sss n$ is globally irreducible \iff it remains irreducible in characteristic $p$ for every prime~$p$. But for other groups this is not the case. In this paper we address globally irreducible representations of the proper double covers of the symmetric groups and the alternating groups $\aaa n$. In \cite{t1} Tiep considered basic spin representations of symmetric and alternating groups and classified the basic spin representations which are composition factors of globally irreducible representations.

One interesting fact about basic spin representations (shown in \cite{wal}) is that they reduce \emph{almost homogeneously} in every characteristic. We say that a representation reduces \emph{homogeneously} in characteristic $p$ if all the composition factors of the $p$-modular reduction are isomorphic. Reducing almost homogeneously is a natural extension of this: a representation reduces almost homogeneously if all the composition factors of its $p$-modular reduction are labelled by the same partition in the standard labelling. For the double cover of $\sss n$, this means that any two composition factors are either isomorphic or obtained from each other by tensoring with the sign representation; for the double cover of $\aaa n$, it means that any two factors are either isomorphic or obtained from each other under the action of the double cover of $\sss n$. 
It can be checked using \cref{tiep2} below that if $M$ is a composition factor of a globally irreducible representation of either an alternating group or a double cover of a symmetric or alternating group, then $M$ reduces almost homogeneously in every characteristic. Therefore in order to classify representations appearing in globally irreducible representations we first classify representations that reduce almost homogeneously in every characteristic. This gives a second reason for studying almost homogeneous reductions. As a by-product, this will also allow us to characterise representations that remain irreducible in every characteristic (this was already known for $\sss n$ and $\aaa n$, but is a new result for the double covers).

Our main result on globally irreducible representations is that, with finitely many exceptions, globally irreducible representations of symmetric and alternating groups and their double covers are either $1$-dimensional or basic spin representations.

For alternating groups we obtain the following result.

\begin{thm}\label{mainan}
Suppose $\la$ is a partition of $n$, and let $M$ be an irreducible $\bbc\aaa n$-module labelled by $\la$. The following are equivalent:
\begin{enumerate}
\item\label{angir}
$M$ appears in a globally irreducible representation;
\item\label{anhom}
the $p$-modular reduction of $M$ is almost homogeneous for every prime $p$;
\item\label{anirr}
the $p$-modular reduction of $M$ is irreducible for every prime $p$;
\item\label{an1dim}
$\la$ or $\la'$ equals $(n)$, $(2,1)$ or $(2,2)$.
\end{enumerate}
In particular if $\aaa n\domby G$ and $N$ is a globally irreducible representation of $G$, then any composition factor of $N\da^G_{\aaa n}$ satisfies the above conditions.
\end{thm}

For any partition $\la$, let $\ant^\la$ or $\ant^\la_\pm$ be the irreducible $\bbc\aaa n$-representations indexed by $\la$. Note that the partitions appearing in (4) are exactly the partitions for which $M=\ant^\la$ or $M=\ant^\la_\pm$ is $1$-dimensional, so this result is almost directly analogous to \cref{mainsn}. However for $\la=(2,1)$ or $(2,2)$ the module $M$ is not itself a globally irreducible representation (as it is not defined over $\Q$), but the representation $\ant^\la_+\oplus \ant^\la_-$ is globally irreducible.

For spin representations of double covers we have the following results. We write $\hsss n^\pm$ and $\haaa n$ for the proper double covers of $\sss n$ and $\aaa n$ (our sign convention for the double covers of $\sss n$ is explained in \cref{doublecoversec}).

\begin{thm}\label{mainhom}
Suppose $\la$ is a strict partition of $n$, and let $M$ be an irreducible spin $\bbc\hsss n^\pm$- or $\bbc\haaa n$-module labelled by $\la$. Then the $p$-modular reduction of $M$ is almost homogeneous for every prime $p$ \iff one of the following occurs:
\begin{enumerate}
\item
$\la=(n)$;
\item
$\la=(2,1)$, $(3,2)$, $(3,2,1)$, $(4,3,2)$, $(4,3,2,1)$, $(5,4,3,2)$ or $(5,4,3,2,1)$.
\end{enumerate}
In particular if $\haaa n\domby G$ and $N$ is a globally irreducible representation of $G$, then any composition factor of $N\da^G_{\haaa n}$ is labeled by one of the above partitions.
\end{thm}

\begin{thm}\label{mains}
Suppose $\la$ is a strict partition of $n$, and let $M$ be an irreducible spin $\bbc\hsss n^\pm$-module labelled by $\la$. Then the $p$-modular reduction of $M$ is irreducible for every prime $p$ \iff one of the following occurs:
\begin{enumerate}
\item
$\la=(n)$, where $n=1$ or $n$ is even;
\item
$\la=(2,1)$, $(3,2)$ or $(3,2,1)$.
\end{enumerate}
\end{thm}

\begin{thm}\label{maina}
Suppose $\la$ is a strict partition of $n$, and let $M$ be an irreducible spin $\bbc\haaa n$-module labelled by $\la$. Then the $p$-modular reduction of $M$ is irreducible for every prime $p$ \iff one of the following occurs:
\begin{enumerate}
\item
$\la=(n)$, where $n=0$, $n=2$ or $n$ is odd;
\item
$\la=(2,1)$, $(4,3,2)$, $(4,3,2,1)$, $(5,4,3,2)$ or $(5,4,3,2,1)$.
\end{enumerate}
\end{thm}

\begin{thm}\label{maingirsp}
Suppose $\la$ is a strict partition of $n$, and let $M$ be an irreducible spin $\bbc\hsss n^+$-module labelled by $\la$. Then $M$ appears in a globally irreducible representation \iff one of the following occurs.
\begin{enumerate}
\item
$\la=(n)$ and one of the following holds:
\begin{enumerate}
\item
$n=8m^2$ with $m\in\Z$,
\item
$n\equiv 2\ppmod{4}$,
\item
$n\equiv 3\ppmod{8}$ and every prime divisor of $n$ is congruent to $3$ or $5$ modulo $8$,
\item
$n\equiv 5\ppmod{8}$ and every prime divisor of $n$ is congruent to $5$ or $7$ modulo $8$.
\end{enumerate}
\item
$\la=(2,1)$, $(5,4,3,2)$ or $(5,4,3,2,1)$.
\end{enumerate}
\end{thm}

\begin{thm}\label{maingirsm}
Suppose $\la$ is a strict partition of $n$, and let $M$ be an irreducible spin $\bbc\hsss n^-$-module labelled by $\la$. Then $M$ appears in a globally irreducible representation \iff one of the following occurs.
\begin{enumerate}
\item
$\la=(n)$ and one of the following holds:
\begin{enumerate}
\item
$n=2(2m+1)^2$ with $m\in\Z$,
\item
$n\equiv 0\ppmod{4}$,
\item
$n\equiv 5\ppmod{8}$ and every prime divisor of $n$ is congruent to $3$ or $5$ modulo $8$,
\item
$n\equiv 7\ppmod{8}$ and every prime divisor of $n$ is congruent to $5$ or $7$ modulo $8$.
\end{enumerate}
\item
$\la=(2,1)$, $(3,2)$ or $(3,2,1)$.
\end{enumerate}
\end{thm}

\begin{thm}\label{maingira}
Suppose $\la$ is a strict partition of $n$, and let $M$ be an irreducible spin $\bbc\haaa n$-module labelled by $\la$. Then $M$ appears in a globally irreducible representation \iff one of the following occurs.
\begin{enumerate}
\item
$\la=(n)$ and one of the following holds:
\begin{enumerate}
\item
$n=(2m+1)^2$ with $m\in\Z$,
\item
$n\equiv 3\ppmod{4}$,
\item
$n=2m^2$ with $m\in\Z_{>0}$ and every prime divisor of $m$ is congruent to $3$ modulo $4$,
\item
$n\equiv 6\ppmod{8}$ and every odd prime divisor of $n$ is congruent to $3$ modulo $4$,
\item
$n=0$ or $4$.
\end{enumerate}
\item
$\la=(2,1)$, $(3,2)$, $(3,2,1)$, $(4,3,2)$, $(4,3,2,1)$, $(5,4,3,2)$ or $(5,4,3,2,1)$.
\end{enumerate}
\end{thm}

\begin{acks}
We thank Sasha Kleshchev for suggesting that we work on this problem, and Pham Huu Tiep for useful conversations. We also thank the referees for comments that helped improve the exposition.
\end{acks}

\section{Background}

\subsection{Globally irreducible representations}\label{gir}

Throughout this section let $G$ be a finite group. Let $V$ be an irreducible $\bbq G$-representation. Let $K:=\End_{\bbq G}(V)$ and $R\subseteq K$ be a maximal order. Further let $\Lambda$ be an $RG$-lattice in $V$, that is $\Lambda$ is the $\bbz$-span of a $\bbq$-basis of $V$ and $\Lambda$ is stable under both $R$ and $G$. Following \cite{gr,t2} we say that $V$ is a \emph{globally irreducible representation} (or GIR) of $G$ if $V\otimes_{\bbq}\bbr$ is irreducible and $\Lambda/I\Lambda$ is irreducible as $(R/I)G$-module for every maximal two-sided ideal $I\subset R$.

We will use the following results on globally irreducible representations. In the following, for any character $\chi$, $\overline{\chi}$ means the complex conjugate character and $\ind(\chi)$ is the Frobenius-Schur indicator of $\chi$. Further, given a prime $p$, we will view complex characters also as Brauer characters (restricting them implicitly to the set of $p'$-elements of $G$).

We will use the following essential results on GIRs.

\begin{propn}\cite[Lemma 2.3]{t2}\label{tiep1}
Let $\chi$ be the character of a GIR of a finite group $G$. Then one of the following holds:
\begin{enumerate}
\item
$\chi$ is absolutely irreducible, $\ind(\chi)=1$ and $K=\bbq$;
\item
$\chi=\psi+\overline{\psi}$ for $\psi$ an absolutely irreducible character and $K=\bbq(\psi)$ is an imaginary quadratic field;
\item
$\chi=2\psi$ for some absolutely irreducible character with $\ind(\psi)=-1$ and $K$ is a definite quaternion algebra.
\end{enumerate}
\end{propn}

\begin{propn}\cite[Proposition 4.2]{gr}\label{gross}
Let $\psi$ be an irreducible complex character of a finite group $G$.
\begin{enumerate}
\item
If $\bbq(\psi)=\bbq$ and $\psi$ is an irreducible Brauer character for all primes $p$, then $\psi$ is the character of a GIR of $G$.
\item
If $\bbq(\psi)$ is an imaginary quadratic field and $\psi$ is an irreducible Brauer character for all primes $p$, then $\psi+\overline{\psi}$ is the character of a GIR of $G$.
\item
If $\bbq(\psi)=\bbq$, $\ind(\chi)=-1$ and for any prime $p$ either $\psi$ is an absolutely irreducible Brauer character or $\psi\equiv\rho+\rho^p\ppmod{p}$ for some absolutely irreducible Brauer character $\rho$ with $\bbf_p(\rho)=\bbf_{p^2}$, then $2\psi$ is the character of a GIR of $G$.
\end{enumerate}
\end{propn}

\begin{propn}\cite[Proposition 2.7]{t2}\label{tiep2}
Let $\chi$ be the character of a GIR of a finite group $G$, let $\psi$ be an absolutely irreducible constituent of $\chi$ and let $p$ be a prime. Then there exists an absolutely irreducible Brauer character $\rho$ such that $\psi\equiv e(\rho_1+\dots+\rho_s)\ppmod{p}$ with $e=1$ or $2$ and $\rho_1,\dots,\rho_s$ distinct conjugates of $\rho$ over $\overline{\bbf_p}$. Moreover if $e=2$ then $K$ is a quaternion algebra and $p$ is ramified in~$R$.
\end{propn}

\subsection{$p$-modular reductions}

Our aim is to study $p$-modular reduction. Given a finite group $G$ and a $\bbc G$-module $M$, the $p$-modular reduction of $M$ is not well-defined up to isomorphism, but its composition factors are, and for each irreducible module $D$ in characteristic $p$ we write $[M:D]$ for the multiplicity of $D$ as a composition factor of a $p$-modular reduction of $M$.


\section{The alternating groups}

In this section we prove our main theorem for the alternating group $\aaa n$. We begin by summarising the classification of irreducible modules for $\sss n$ and $\aaa n$.

\subsection{Representations in characteristic $0$}

It is well known that irreducible representations of $\sss n$ over $\C$ are given by the \emph{Specht modules} $\sns^\la$ labelled by partitions of $n$ (see for example \cite{JamesBook,JK}). Moreover, it is also well-known that $\sns^\la\otimes\sgn\cong\sns^{\la'}$; see for example \cite[2.1.8]{JK}. This allows us to describe the irreducible representations of $\aaa n$ over $\C$ (see for example \cite[\S2.5]{JK}).

\begin{thm}\label{snirred0}\indent
\begin{enumerate}
\vspace{-\topsep}
\item
The Specht modules $\sns^\la$ give a complete irredundant list of irreducible $\bbc\sss n$-modules as $\la$ ranges over the partitions of $n$.
\item
For each partition $\la$ of $n$ with $\la\not=\la'$ there is a self-associate irreducible $\bbc\aaa n$-module $\ant^\la$, such that
\[
\sns^\la\da^{\sss n}_{\aaa n}\cong \ant^\la,\qquad\ant^\la\ua_{\aaa n}^{\sss n}\cong \sns^\la\oplus \sns^{\la'}.
\]
\item
For each partition $\la$ of $n$ with $\la=\la'$ there is an associate pair of irreducible $\bbc\aaa n$-modules $\ant^\la_+$, $\ant^\la_-$, such that
\[
\sns^\la\da^{\sss n}_{\aaa n}\cong \ant^\la_+\oplus\ant^\la_-,\qquad\ant^\la_\pm\ua_{\aaa n}^{\sss n}\cong \sns^\la.
\]
\item
The modules $\ant^\la$ (for $\la\neq\la'$) and $\ant^\la_\pm$ (for $\la=\la'$) together give a complete list of irreducible $\bbc\aaa n$-modules. The only non-trivial isomorphisms between these modules are those of the form $\ant^\la\cong\ant^{\la'}$ for $\la\not=\la'$.
\end{enumerate}
\end{thm}

For simplicity reasons, we will write $\ant^\la_\ast$ for either $\ant^\la$ if $\la\neq\la'$, or one of the modules $\ant^\la_\pm$ if $\la=\la'$.

\subsection{Representations in positive characteristic}\label{snirred}

Now let $p$ be a prime. The irreducible representations of $\sss n$ in characteristic $p$ are labelled by the set $\Par_p(n)$ of \emph{$p$-regular} partitions of $n$, that is partitions where no part is repeated $p$ or more times. For each $\la\in\Par_p(n)$, James \cite[Section 11]{JamesBook} constructs a module $\snd^\la$ such that the following holds.

\begin{thm}[\textup{\textbf{\cite[Theorem 11.5]{JamesBook}}}]
The modules $\snd^\la$ for $\la\in\Par_p(n)$ give a complete irredundant list of irreducible $\overline{\F_p}\sss n$-modules.
\end{thm}

If $\la$ is a $p$-regular partition then $\snd^\la\otimes\sgn$ is also irreducible, so there is a $p$-regular partition $\la^\mull$ such that $\snd^\la\otimes\sgn\cong\snd^{\la^\mull}$. The function $\la\mapsto\la^\mull$ is called the \emph{Mullineux map}, and admits several combinatorial descriptions (which we shall not need here). If $p=2$ then by definition $\la^\mull=\la$ for every $\la\in\Par_2(n)$.

We can now describe the classification of irreducible representations of $\aaa n$ in characteristic $p$. For odd $p$ this was given by Ford \cite{ford}, while for $p=2$ the classification was obtained by Benson \cite{ben}. Their results can be combined in the following \lcnamecref{anrep}. (Part 1 of the \lcnamecref{anrep} uses the fact that (writing $\triv$ for the trivial module for any group) 
$[\triv\ua^{\sss n}_{\aaa n}]=[\triv]+[\sgn]$ in the Grothendieck group of $\sss n$ over any field.)

\begin{thm}\label{anrep}
For each prime $p$, there is a subset $\Parinv_p(n)$ of $\Par_p(n)$ such that the following hold. 
\begin{enumerate}
\item
For each $\la\in\Par_p(n)\setminus\Parinv_p(n)$ there is a self-associate irreducible $\overline{\F_p}\aaa n$-module $\ane^\la$, such that
\[
\snd^\la\da^{\sss n}_{\aaa n}\cong \ane^\la,\qquad[\ane^\la\ua_{\aaa n}^{\sss n}]=[\snd^\la]+[\snd^{\la^\mull}]\text{ (in the Grothendieck group of $\overline{\bbf_p}\sss n$)}.
\]
\item
For each $\la\in\Parinv_p(n)$ there is an associate pair of irreducible $\overline{\F_p}\aaa n$-modules $\ane^\la_+$, $\ane^\la_-$ such that
\[
\snd^\la\da^{\sss n}_{\aaa n}\cong \ane^\la_+\oplus\ane^\la_-,\qquad\ane^\la_\pm\ua_{\aaa n}^{\sss n}\cong \snd^\la.
\]
\item
The modules $\ane^\la$ (for $\la\in\la\in\Par_p(n)\setminus\Parinv_p(n)$) and $\ane^\la_\pm$ (for $\la\in\Parinv_p(n)$) give a complete list of irreducible spin $\overline{\F_p}\aaa n$-modules. The only non-trivial isomorphisms between these modules are exactly those of the form $\ane^\la\cong\ane^{\la^\mull}$ for $p\not=2$ and $\la\in\Par_p(n)\setminus\Parinv_p(n)$.
\end{enumerate}
\end{thm}

In fact when $p$ is odd, $\Parinv_p(n)$ is just the set $\lset{\la\in\Par_p(n)}{\la^\mull=\la}$ of fixed points of the Mullineux map. The set $\Parinv_2(n)$ also admits a simple combinatorial description, but we will not need this.

In view of the above result, when $p$ is understood we say that a $p$-regular partition $\la$ \emph{splits} if $\la\in\Parinv_p(n)$, as this is exactly the situation where $\snd^\la\da^{\sss n}_{\aaa n}$ is reducible.

We define $\ane^\la_\ast$ similarly to $\ant^\la_\ast$.

%

\subsection{Proof of the main result for alternating groups}

In this subsection we prove \cref{mainan}. Recall from the introduction that a module $M$ for $\overline{\bbf_p}\aaa n$ is \emph{homogeneous} if its composition factors are all isomorphic, or \emph{almost homogeneous} if its composition factors can all be labelled by the same partition; that is, either $M$ is homogeneous or there is $\la\in\Parinv_p(n)$ such that each composition factor of $M$ is isomorphic to $\ane_+^\la$ or $\ane_-^\la$.

If $M$ is a $\bbc\aaa n$-module and $p$ is a prime, then we say that $M$ is (almost) homogeneous in characteristic $p$ if a $p$-modular reduction of $M$ is (almost) homogeneous.

To prove \cref{mainan} we need to recall James's regularisation theorem. For this, recall the dominance order $\dom$ on partitions of $n$: $\mu\dom\la$ if $\mu_1+\dots+\mu_r\gs\la_1+\dots+\la_r$ for every $r$.

\begin{thm}[\textup{\textbf{\cite[Theorem A]{jreg}}}]\label{regn}
Suppose $\la$ is a partition and $p$ a prime. Then there is a $p$-regular partition $\la^\reg$ such that $[\sns^\la:\snd^{\la^\reg}]=1$ while $\mu\dom\la^\reg$ for any composition factor $\snd^\mu$ of $\sns^\la$.
\end{thm}

%

We start by proving a fixed characteristic version of the equivalence of conditions \eqref{anhom} and \eqref{anirr} in \cref{mainan}. We prove only one direction, since the other holds by definition.

\begin{thm}\label{homogan}
Let $\la$ be a partition and $p$ a prime. If $\ant^\la_\ast$ is almost homogeneous in characteristic $p$ then it is irreducible in characteristic $p$.
\end{thm}

\begin{proof}
Assume that $\ant^\la_\ast$ is almost homogeneous. Then there is a $p$-regular partition $\mu$ such that (in the Grothendieck group of $\overline{\bbf_p}\aaa n$) either $[\ant^\la_\ast]=a[\ane^\mu]$ with $\mu\not\in\Parinv_p(n)$ or $[\ant^\la_\ast]=a[\ane^\mu_+]+b[\ane^\mu_-]$ with $\mu\in\Parinv_p(n)$. Now \cref{anrep} gives $[\ant^\la_\ast\ua^{\sss n}]=a[\snd^\mu]+a[\snd^{\mu^\mull}]$ or $(a+b)[\snd^\mu]$.

In characteristic $0$, \cref{snirred0} shows that $\sns^\la$ appears as a composition factor of $\ant^\la_\ast\ua^{\sss n}$, and therefore in characteristic $p$ every composition factor of $\sspe\la$ is a composition factor of $\ant^\la_\ast\ua^{\sss n}$. So every composition factor of $\sspe\la$ in characteristic $p$ is either $\jms\mu$ or $\jms{\mu^\mull}$. Since $\snd^{\la^\reg}$ is a composition factor of $\sns^\la$ with multiplicity $1$ by \cref{regn}, it follows that $[\sns^\la]=[\snd^{\la^\reg}]+c[\snd^{(\la^\reg)^\mull}]$, with $c=0$ if $\la^\reg=(\la^\reg)^\mull$.

If $c=0$, then $\sns^\la$ is irreducible in characteristic $p$, and hence so is $\ant^\la_\ast$, by \cite[Proposition 2.11]{mfonaltred}. So suppose $c\gs1$. Then $\la^\reg\not=(\la^\reg)^\mull$, and in particular $p\not=2$. In addition, \cref{regn} gives $(\la^\reg)^\mull\doms\la^\reg$. Now consider $\sspe\la\otimes\sgn$. As noted above, $\sspe\la\otimes\sgn\cong\sspe{\la'}$ in characteristic $0$, and therefore $[\sspe\la\otimes\sgn]=[\sspe{\la'}]$ in characteristic $p$. Hence
\[
[\sns^{\la'}]=c[\snd^{\la^\reg}]+[\snd^{(\la^\reg)^\mull}].
\]
Since $(\la^\reg)^\mull\doms\la^\reg$ we deduce (again using \cref{regn}) that $(\la')^\reg=\la^\reg$ and $c=1$. In particular $[\sns^\la]=[\sns^{\la'}]$, so that $\la=\la'$ by \cite[Theorem 1.1.1(i)]{wildon}. Now again \cite[Proposition 2.11]{mfonaltred} shows that $\ant^\la_\ast$ is irreducible in characteristic $p$.
\end{proof}

We next consider the action of the Galois group on the set of irreducible representations of $\aaa n$.

\begin{lemma}\label{galoisan}
Let $\la\in\Par_p(n)$. Then the set $\{\ane^\la\}$ or $\{\ane^\la_\pm\}$ is closed under the Galois action of $\ol{\bbf_p}$.
\end{lemma}

\begin{proof}
By \cite[Theorem 11.5]{JamesBook} $\snd^\la$ can be defined over $\F_p$, so the set of composition factors of $\jms\la\da_{\aaa n}$ is closed under $\ol{\bbf_p}$-conjugation. The lemma then follows by \cref{anrep}.
\end{proof}

We are now ready to prove \cref{mainan}.

\begin{proof}[Proof of \cref{mainan}]
\begin{description}
\imps{angir}{anhom}
Let $p$ be a prime, and suppose $\jms\mu$ is a composition factor of $\sspe\la$ in characteristic $p$. By definition $\snd^\mu\da_{\aaa n}$ is isomorphic to $\ane^\mu$ or $\ane^\mu_+\oplus \ane^\mu_-$, and so $\ane^\mu$ or $\ane^\mu_\pm$ is a composition factor of $\ant^\la_\ast$ in characteristic $p$. By \cref{galoisan} the set of composition factors of $\jms\mu\da_{\aaa n}$ is closed under $\ol{\bbf_p}$-conjugation. 
Since $\ant^\la_\ast$ appears in a GIR, it then follows by \cref{tiep2} that the only possible composition factors of $\ant^\la_\ast$ are of the form $\ane^\mu_\ast$, so $\ant^\la_\ast$ is almost homogeneous.

\imps{anhom}{anirr}
This holds by \cref{homogan}.

\imps{anirr}{an1dim}
By \cite[Theorem 8.1]{mfaltredproof} if $\ant^\la_\ast$ is irreducible in every characteristic, then $\ant^\la_\ast$ is $1$-dimensional. It is then an easy exercise with the hook-length formula to see that $\la$ or $\la'$ is one of $(n)$, $(2,1)$ or $(2^2)$.
%
%
%
\imps{an1dim}{angir}
If $\la=(n)$ then $\ant^\la$ is the trivial representation. In particular its character is defined over $\Q$. So by \cref{gross}(1) $\ant^\la$ is a GIR. If $\la=(2,1)$ or $(2,2)$ then the character field of $\ant^\la_\pm$ is $\Q(\sqrt{-3})$. Furthermore, $\ant^\la_\pm$ remains irreducible in every characteristic because it is $1$-dimensional. So $\ant^\la_\pm$ appears in a GIR by \cref{gross}(2).
\end{description}

Assume now that $G$ is a group containing $\aaa n$ as normal subgroup, that $N$ is a GIR of $G$ and that $M$ appears in $N\da^G_{\aaa n}$. Then by \cite[Proposition 2.8]{t2}, for any prime $p$, composition factors of the reduction modulo $p$ of $M$ are conjugate to each other under $\ol{\bbf_p}$ and $G$. Since $\Aut(\aaa n)\leq \sss n$ for $n\not=6$, the last statement follows by \cref{anrep,galoisan} unless $n=6$. In this last case it follows by analysis of decomposition matrices.
\end{proof}

\section{Double covers of the alternating and symmetric groups}

\subsection{Definition of double covers}\label{doublecoversec}

Let $\sss n$ denote the symmetric group of degree $n$, and $\aaa n$ the alternating group. Double covers of these groups were discovered by Schur \cite{schu} in the study of projective representations. Let $\hsss n^+$ denote the group with generators $s_1,\dots,s_{n-1},z$, subject to the relations
\[
s_i^2=z,\quad z^2=1,\quad s_iz=zs_i,\quad s_is_js_i=s_js_is_j\ \text{ if }j=i+1,\qquad s_is_j=s_js_iz\ \text{ if }j>i+1.
\]
The group $\hsss n^-$ is defined in the same way, but with the relation $s_i^2=1$ in place of $s_i^2=z$. The groups $\hsss n^\pm$ are double covers of $\sss n$, and are \emph{Schur covers} of $\sss n$ provided $n\gs4$. ($\hsss n^+$ is the group denoted $\tilde S_n$ in \cite{hohum,stem}, while $\hsss n^-$ is denoted $\hat S_n$ in \cite{hohum} and $\tilde S_n'$ in \cite{stem}. Tiep \cite{t1} uses the opposite sign convention to ours: the group $2^\pm\mathbb S_n$ in \cite{t1} is our $\hsss n^\mp$.)

To prove \cref{mainhom,mains} we will not need to distinguish between $\hsss n^+$ and $\hsss n^-$, so we will use the notation $\hsss n$ to mean either of these groups, until the end of \cref{alhomsec}. To prove \cref{maingirsp,maingirsm} in \cref{spingirsec} we will have to distinguish between the two double covers. If we need to distinguish between generators of $\hsss n^+$ and $\hsss n^-$, we will write $s_{i,\pm}$ instead of $s_i$.

We write $\haaa n$ for the pre-image of $\aaa n$ under the covering map $\hsss n\to\sss n$. Then $\haaa n$ is a double cover of $\aaa n$, and is a Schur cover of $\aaa n$ provided $n\gs4$ and $n\neq6,7$.

We will also need to consider lifts of Young subgroups: if $\al$ is a composition of $n$ then we define $\hsss{\al}^\pm$ to be the subgroup of $\hsss n^\pm$ generated by $z$ and all $s_i$ with $i\not=\sum_{k=1}^j\al_k$ for any $j\geq 1$. When considering explicit $\al$ we will omit the parentheses. For example, $\hsss{4,2,3}^\pm=\langle z,s_1,s_2,s_3,s_5,s_7,s_8\rangle$. We also define $\haaa\al=\hsss\al^\pm\cap\haaa n$ to be the corresponding subgroup of $\haaa n$.

\subsection{Combinatorics of strict and $p$-strict partitions}

Now we describe the combinatorics of partitions that underpins the representations of $\hsss n$ and $\haaa n$. 

Suppose $\la$ is a partition. We write $\len\la$ for the \emph{length} of $\la$, i.e.\ the largest $r$ for which $\la_r>0$. We say that $\la$ is \emph{strict} if $\la_r>\la_{r+1}$ for all $1\ls r<\len\la$ (so ``strict'' is just a synonym for the term ``$2$-regular'' used in \cref{snirred}). A partition $\la$ is \emph{even} if it has an even number of positive even parts, and \emph{odd} otherwise. Given two partitions $\la$ and $\mu$ and a natural number $n$, we may write $\la+n\mu$ for the partition $(\la_1+n\mu_1,\la_2+n\mu_2,\dots)$. We also define $\la\sqcup\mu$ to be the partition whose parts are the combined parts of $\la$ and $\mu$, written in decreasing order.

The \emph{Young diagram} of a partition $\la$ is the set
\[
[\la]=\lset{(r,c)\in\bbn^2}{c\ls\la_r}
\]
whose elements are called the \emph{nodes} of $\la$. We draw Young diagrams as arrays of boxes using the English convention, in which $r$ increases down the page and $c$ increases from left to right.
 
If $\la$ is a strict partition, then a node $(r,c)\in[\la]$ is \emph{removable} if $[\la]\sm\{(r,c)\}$ is also the Young diagram of a strict partition. A pair $(r,c)\notin[\la]$ is an \emph{addable node} of $\la$ if $[\la]\cup\{(r,c)\}$ is the Young diagram of a strict partition.

\begin{wng}
The definition of addable and removable nodes we have used here is not universal: sometimes for dealing with representations in characteristic $p$, a more liberal definition of addable and removable nodes is used which depends on $p$. But because we allow $p$ to vary, we stick with the more restrictive definition above.
\end{wng}

We now define residues and ladders for a given prime $p$. For $p=2$ the \emph{$2$-residue} of a node $(r,c)$ is $0$ if $c\equiv 0\text{ or }1\ppmod{4}$, and $1$ otherwise. So the $2$-residue of a node depends only on its column, and the residues follow the repeating pattern
\[
0,1,1,0,0,1,1,0,\dots
\]
from left to right. 

\emph{Ladders} for $p=2$ were introduced by Bessenrodt and Olsson \cite{bo}, and are defined as follows. For each $k\gs0$, we define the $k$th ladder to be the set of nodes
\[
\call_k=\rset{(r,c)\in\bbn^2}{\left\lfloor\mfrac{c}2\right\rfloor+2(r-1)=k}.
\]
For example the first ladders can be illustrated in the following diagram, where we label all the nodes in $\call_k$ with $k$.
\[
\gyoung(01122334455^\hf^2\hdts,2334455^\hf^2\hdts,455^\hf^2\hdts)
\]

Now let $p=2l+1$ be an odd prime. The \emph{$p$-residue} of a node $(r,c)$ is the smaller of the residues of $c-1$ and $-c$ modulo $p$. So again the $p$-residue of a node depends only on its column, and the residues follow the repeating pattern
\[
0,1,\dots,l-1,l,l-1,\dots,1,0,0,1,\dots,l-1,l,l-1,\dots,1,0,\dots
\]
from left to right. We say that a partition $\la$ is \emph{$p$-even} if it has an even number of nodes of non-zero residue, and \emph{$p$-odd} otherwise.

Ladders for $p$ odd were introduced by Brundan and Kleshchev \cite{bkreg}, and are defined as follows. For each $k\gs0$, we define the $k$th ladder to be the set of nodes
\[
\call_k=\rset{(r,c)\in\bbn^2}{\left\lfloor\mfrac{(p-1)c}p\right\rfloor+(p-1)(r-1)=k}.
\]
For example, when $p=5$, the ladders can be illustrated in the following diagram, where we label all the nodes in $\call_k$ with $k$.
\[
\gyoung(012344567889^\hf^2\hdts,4567889^\hf^2\hdts,89^\hf^2\hdts)
\]
The ladder $\call_k$ depends on the prime $p$ as well as on $k$, but $p$ will always be clear from the context. For any $p$, if $k_1<k_2$, then we say that the ladder $\call_{k_2}$ is \emph{longer} than $\call_{k_1}$.


For any prime and any residue $i$, an \emph{$i$-node} means a node of residue $i$.

When $p$ is odd, we need to recall some more definitions. We say that a partition $\la$ is \emph{$p$-strict} if for every $r$ either $\la_r>\la_{r+1}$ or $p\mid\la_r$. A $p$-strict partition $\la$ is \emph{$p$-restricted} if for each $r$ either $\la_r<\la_{r+1}+p$ or $\la_r=\la_{r+1}$ and $p\mid\la_r$.

\subsection{Representations in characteristic $0$}\label{char0}

Now we describe the classification of irreducible representations of $\hsss n$ and $\haaa n$. On an irreducible module for $\hsss n$ or $\haaa n$ over any field, the central element $z$ must act as either $1$ or $-1$. Modules on which $z$ acts as $1$ reduce to modules for $\sss n$ or $\aaa n$, while modules on which $z$ acts as $-1$ are called \emph{spin} modules. By \cite[p.\ 93]{stem} absolutely irreducible spin representations of $\hsss n^+$ and $\hsss n^-$ are essentially the same, though their characters are not (one only has to adjust the action of the generators by a scalar).

If $M$ is a module for $\hsss n$ (over any field), the \emph{associate} module is obtained by tensoring with the one-dimensional sign module $\sgn$ (on which each $s_i$ acts as $-1$, and $z$ acts as $1$). If $M$ is a module for $\haaa n$, the associate module is obtained by conjugating the action of each element of $\haaa n$ by an odd element of $\hsss n$.

The irreducible spin representations of $\hsss n$ and $\haaa n$ over $\bbc$ were classified by Schur. (In fact Schur only found the irreducible characters; the modules themselves were constructed -- at least for $\hsss n$ -- by Nazarov \cite{naz}.) Schur's classification can be stated as follows. 

\needspace{3em}
\begin{thm}\label{schurclassn}\indent
\begin{enumerate}
\vspace{-\topsep}
\item
For each even strict partition $\la$ of $n$, there is a self-associate irreducible $\bbc\hsss n$-module $\spe\la$. For each odd strict partition of $n$ there is an associate pair of irreducible spin $\bbc\hsss n$-modules $\spe\la_+$, $\spe\la_-$. The modules constructed in this way give a complete irredundant list of irreducible spin $\bbc\hsss n$-modules.
\item
For each odd strict partition $\la$ of $n$, there is a self-associate irreducible $\bbc\haaa n$-module $\ape\la$. For each even strict partition of $n$ there is an associate pair of irreducible spin $\bbc\haaa n$-modules $\ape\la_+$, $\ape\la_-$. The modules constructed in this way give a complete irredundant list of irreducible spin $\bbc\haaa n$-modules.
\item
If $\la$ is an even strict partition of $n$ then $\spe\la\da^{\hsss n}_{\haaa n}\cong \ape\la_+\oplus\ape\la_-$ and $\ape\la_\pm\ua_{\haaa n}^{\hsss n}\cong \spe\la$. If $\la$ is an odd strict partition of $n$ then $\spe\la_\pm\da^{\sss n}_{\aaa n}\cong \ape\la$ and $\ape\la\ua_{\aaa n}^{\sss n}\cong \spe\la_+\oplus \spe{\la}_-$.
\end{enumerate}
\end{thm}

If $\la$ is a strict partition, we will write $\spes\la$ to mean $\spe\la$ if $\la$ is even, or either of the modules $\spe\la_\pm$ if $\la$ is odd. We use the notation $\apes\la$ similarly.

We also need notation for irreducible spin representations of $\hsss\al$ for a composition $\al=(\al_1,\ldots,\al_h)$. In \cite[\S4]{stem}, Stembridge introduces the \emph{reduced Clifford product} $(\spe{\la^1}\otimes\cdots\otimes\spe{\la^h})_\ast$, where $\la^j$ is a strict partition of $\al_j$ for each $j$. See in particular \cite[(4.3)]{stem} for the construction and \cite[Proposition 4.2]{stem} for the characters of these modules. In \cite[Theorem 4.3]{stem} it is shown that these representations are exactly the irreducible representations of $\hsss\al$. In this paper we will need them only in the case where the $\la^j$ are all even partitions.

\subsection{Representations of $\hsss n$ and $\haaa n$ in positive characteristic}\label{oddsec}

In characteristic $2$, the central element $z$ acts as $1$ on every irreducible module for $\hsss n$ or $\haaa n$, which means that the irreducible modules for $\hsss n$ reduce to modules for $\sss n$ (and similarly for $\haaa n$ and $\aaa n$). This means that when a spin representation of $\bbc\hsss n$ or $\bbc\haaa n$ is reduced modulo $2$, the composition factors of the resulting module are all modules of the form $\snd^\la$ or $\ane^\la$ or $\ane^\la_\pm$ introduced in \cref{snirred}.

When $p$ is odd, however, the composition factors of the reductions modulo $p$ of spin representations are still spin representations. The representation theory of $\hsss n$ and $\haaa n$ over a field of odd characteristic has been developed over a long period. Labelling sets for irreducible spin modules in odd characteristic were found by Brundan and Kleshchev. We summarise the results we need (with minor changes to notation) as explained in Kleshchev's book \cite[\S22.3]{kleshbook}.

A typical modern approach in this subject is to regard the group algebra of $\hsss n$ as a \emph{superalgebra} (i.e.\ a $\bbz/2\bbz$-graded algebra), with the generators $s_1,\dots,s_{n-1}$ in odd degree and $z$ in even degree, and to work with irreducible supermodules. Then to derive results on irreducible modules, one can use the well-understood relationship between modules and supermodules. In particular, \cref{schurclassn}(1) can be expressed by saying that there is an irreducible spin $\bbc\hsss n$-supermodule $\bt\la$ for each strict partition $\la$ of $n$. As modules (i.e.\ forgetting the $\bbz/2\bbz$-grading) $\bt\la$ coincides with $\spe\la$ if $\la$ is even, or with $\spe\la_+\oplus \spe\la_-$ if $\la$ is odd. We also similarly define $\bta\la$ to be either of $\ape\la$ or $\ape\la_+\oplus\ape\la_-$.

Now we fix an odd prime $p$, and suppose $\bbf$ is a splitting field for $\hsss n$ of characteristic $p$. 
For each $p$-restricted $p$-strict partition $\la$ of $n$, Kleshchev defines the following:
\begin{itemize}
\item
a supermodule $\bd\la$ for $\bbf\hsss n$;
\item
a module $\smp\la$ for $\bbf\hsss n$ and modules $\amp\la_\pm$ for $\bbf\haaa n$, if $\la$ is $p$-even;
\item
modules $\smp\la_\pm$ for $\bbf\hsss n$ and a module $\amp\la$ for $\bbf\haaa n$, if $\la$ is $p$-odd.
\end{itemize}
These modules provide a classification of irreducible spin (super)modules, as in the following theorem, which is a combination of Theorem 22.3.1 and p.267 in \cite{kleshbook}.

\needspace{3em}
\begin{thm}\label{spinirred}\indent
\begin{enumerate}
\vspace{-\topsep}
\item
The modules $\smp\la$ for $\la$ a $p$-even $p$-restricted $p$-strict partition of $n$ and $\smp\la_\pm$ for $\la$ a $p$-odd $p$-restricted $p$-strict partition of $n$ give a complete irredundant list of irreducible spin $\F\hsss n$-modules.
\item
The modules $\amp\la_\pm$ for $\la$ a $p$-even $p$-restricted $p$-strict partition of $n$ and $\amp\la$ for $\la$ a $p$-odd $p$-restricted $p$-strict partition of $n$ give a complete irredundant list of irreducible spin $\F\haaa n$-modules.
\item
The modules $\bd\la$ for $\la$ a $p$-strict $p$-restricted partition give a complete irredundant list of irreducible spin $\F\haaa n$-supermodules.
\item
If $\la$ is a $p$-even $p$-restricted $p$-strict partition $\la$ of $n$ then $\smp\la\da^{\hsss n}_{\haaa n}\cong \amp\la_+\oplus\amp\la_-$ and $\amp\la_\pm\ua_{\haaa n}^{\hsss n}\cong \smp\la$. If $\la$ is a $p$-odd $p$-restricted $p$-strict partition $\la$ of $n$ then $\smp\la_\pm\da^{\sss n}_{\aaa n}\cong \amp\la$ and $\amp\la\ua_{\aaa n}^{\sss n}\cong \smp\la_+\oplus \smp{\la}_-$.
\item If $\la$ is a $p$-even $p$-restricted $p$-strict partition $\la$ of $n$ then, as modules, $\bd\la\cong\smp\la$. If $\la$ is a $p$-odd $p$-restricted $p$-strict partition $\la$ of $n$ then, as modules, $\bd\la\cong\smp\la_+\oplus\smp\la_-$.
\end{enumerate}
\end{thm}

For any $p$-restricted $p$-strict partition of $n$ we also define a module $\be\la$ of $\haaa n$ by $\be\la:=\amp\la$ if $\la$ is $p$-even, or $\be\la:=\amp\la_+\oplus\amp\la_-$ if $\la$ is $p$-odd. Further we define $\smp\la_\ast$ to be either $\smp\la$ or either of $\smp\la_\pm$, and define $\amp\la_\ast$ similarly.



\section{Homogeneous reductions for double covers}\label{alhomsec}

In this section we study (almost) homogeneous reductions and prove \cref{mainhom,mains,maina}. As with modules for $\aaa n$, we say that a (super)module $M$ for $\hsss n$ or $\haaa n$ is \emph{homogeneous} if its composition factors are all isomorphic, or \emph{almost homogeneous} if its composition factors are all labelled by the same partition. If $M$ is defined over $\bbc$ and $p$ is a prime, then we say that $M$ is (almost) homogeneous in characteristic $p$ if a $p$-modular reduction of $M$ is (almost) homogeneous.

If $M$ is a supermodule, then we say that $M$ is homogeneous in characteristic $p$ if the composition factors of a $p$-modular reduction of $M$ (as a supermodule) are isomorphic. (For $p=2$, there is a one-to-one correspondence between simple modules and simple supermodules, so this condition is equivalent to saying that $M$ is homogeneous as a module.)

\begin{lemma}\label{inhom}
Suppose $\la$ is a strict partition of $n$ and $p$ is a prime. Then the following are equivalent:
\begin{itemize}
\item $\bt\la$ is homogeneous in characteristic $p$;
\item $\spes\la$ is almost homogeneous in characteristic $p$;
\item $\apes\la$ is almost homogeneous in characteristic $p$.
\end{itemize}
\end{lemma}

\begin{proof}
Consider first the case $p=2$, and recall from \cref{snirred} that we write $\Parinv_2(n)$ for the set of partitions $\mu$ of $n$ such that the restriction of $\snd^\mu$ to $\aaa n$ is reducible. The relationship between irreducible modules and supermodules for $\hsss n$, and between irreducible modules for $\hsss n$ and $\haaa n$, means that (writing $[\bt\la:\snd^\mu]$ for the composition multiplicity of $\snd^\mu$ in $\bt\la$ as a module)
\begin{alignat*}2
[\spe\la:\snd^\mu]=2[\ape\la_\pm:\ane^\mu]&=[\bt\la:\snd^\mu]&\quad&\text{if $\la$ is even and $\mu\notin\Parinv_2(n)$,}
\\
[\spe\la:\snd^\mu]=[\ape\la_\pm:\ane^\mu_+]+[\ape\la_\pm:\ane^\mu_-]&=[\bt\la:\snd^\mu]&\quad&\text{if $\la$ is even and $\mu\in\Parinv_2(n)$,}
\\
[\spe\la_\pm:\snd^\mu]=[\ape\la:\ane^\mu]&=\tfrac12[\bt\la:\snd^\mu]&\quad&\text{if $\la$ is odd and $\mu\notin\Parinv_2(n)$,}
\\
[\spe\la_\pm:\snd^\mu]=[\ape\la:\ane^\mu_\pm]&=\tfrac12[\bt\la:\snd^\mu]&\quad&\text{if $\la$ is odd and $\mu\in\Parinv_2(n)$.}
\end{alignat*}

The proof in odd characteristic $p$ is similar. Given a restricted $p$-strict partition $\mu$, we write $[\bt\la:\bd\mu]$ for the multiplicity of $\bd\mu$ as a (super)composition factor of a $p$-modular reduction of $\bt\la$. Then
\begin{alignat*}2
[\spe\la:\smp\mu]=[\ape\la_\pm:\amp\mu_+]+[\ape\la_\pm:\amp\mu_-]&=[\bt\la:\bd\mu]&\quad&\text{if $\la$ is even and $\mu$ is $p$-even,}
\\
[\spe\la:\smp\mu_\pm]=[\ape\la_\pm:\amp\mu]&=[\bt\la:\bd\mu]&\quad&\text{if $\la$ is even and $\mu$ is $p$-odd,}
\\
[\spe\la_\pm:\smp\mu]=[\ape\la:\amp\mu_\pm]&=\tfrac12[\bt\la:\bd\mu]&\quad&\text{if $\la$ is odd and $\mu$ is $p$-even,}
\\
[\spe\la_\pm:\smp\mu_+]+[\spe\la_\pm:\smp\mu_-]=[\ape\la:\amp\mu]&=[\bt\la:\bd\mu]&\quad&\text{if $\la$ is odd and $\mu$ is $p$-odd.}
\end{alignat*}
\end{proof}


In order to exploit \cref{inhom}, we use the following \lcnamecref{laddlem}.

\begin{propn}\cite[Proposition 4.10]{fm}\label{laddlem}
Suppose $p=2l+1$ is an odd prime, and $\la$ is a strict partition of $n$. Suppose that there is some residue $i\in\{0,\dots,l\}$ such that $\la$ has a removable $i$-node and an addable $i$-node in a longer ladder. Then $\bt\la$ is inhomogeneous in characteristic $p$.
\end{propn}


We deduce the following useful corollary.

 \begin{cory}\label{addrem}
 Suppose $\la$ is a strict partition and $p$ is an odd prime, and there are $r,s\in\bbn$ with $r<s$ such that:
 \begin{itemize}
 \item
 $\la$ has both addable and removable nodes in rows $r$ and $s$; and
 \item
 $\la_r+\la_s$ is divisible by $p$ with $\la_r-\la_s\neq p(s-r)$.
 \end{itemize}
 Then $\bt\la$ is inhomogeneous in characteristic $p$.
 \end{cory}
 
 \begin{proof}
 The fact that $p\mid\la_r+\la_s$ means that the addable node in row $r$ has the same residue as the removable node in row $s$, and that the removable node in row $r$ has the same residue as the addable node in row $s$. If $\la_r-\la_s>p(s-r)$, then the addable node in row $r$ lies in a longer ladder than the removable node in row $s$, and \cref{laddlem} gives the result. On the other hand, if $\la_r-\la_s<p(s-r)$, then the addable node in row $s$ lies in a longer ladder than the removable node in row $r$, and again \cref{laddlem} applies.
 %
 %
 %
 \end{proof}
 
 For $p=2$ we have the following similar statement, which holds with the same argument as the previous result, using \cite[Proposition 4.17]{mfspin2alt}.
 
\begin{cory}\label{addrem2}
 Suppose $\la$ is a strict partition, and there are $r,s\in\bbn$ with $r<s$ such that:
 \begin{itemize}
 \item
 $\la$ has both addable and removable nodes in rows $r$ and $s$; and
 \item
 $\la_r+\la_s$ is divisible by $4$ and $\la_r-\la_s\neq 4(s-r)$.
 \end{itemize}
 Then $\bt\la$ is inhomogeneous in characteristic $2$.
 \end{cory}

The next lemma studies the action of the Galois group on the sets of irreducible representations of $\hsss n$ or $\haaa n$.

\begin{lemma}\label{galoisspin}
Let $p$ be odd and $\mu$ be a $p$-restricted $p$-strict partition. Then the set $\{\smp\mu\}$ or $\{\smp\mu_\pm\}$ is closed under the Galois action of $\ol{\bbf_p}$. The same holds for the set $\{\amp\mu_\pm\}$ or $\{\amp\mu\}$.
\end{lemma}

\begin{proof}
Notice that the characters of $\bt\la$ and $\bta\la$ are integer valued (this is most easily seen from Morris's analogue of the Murnaghan--Nakayama formula, as given by Hoffman and Humphreys \cite[Theorems 8.7 and 10.1]{hohum}). Recall from \cref{char0} that $\spe\la\cong\spe\la\otimes\sgn$ and $\spe\la_+\cong \spe\la_-\otimes\sgn$, where $\sgn$ is the sign representation of $\sss n$, and similarly that $\ape\la\cong\ape\la^\si$ and $\ape\la_+\cong\ape\la_-^\si$ for any $\si\in\hsss n\sm\haaa n$. From \cref{oddsec} similar formulas holds for $\smps\mu$ and $\amps\mu$ in odd characteristic. It follows from these properties that (in the Grothendieck group) $[\bd\mu]$ can be written as a $\Q$-linear combination of the modules $[\bt\la]$, and similarly for $[\be\mu]$. So the Brauer characters of $\bd\mu$ and $\be\mu$ are rational (and then also integer) valued. The lemma follows.
\end{proof}

We are now ready to prove our main results on homogeneous and irreducible reductions for double covers.

\begin{proof}[Proof of \cref{mainhom}]
In view of \cref{inhom}, to prove the first part we just need to show that $\bt\la$ is homogeneous in every characteristic if and only if $\la=(n)$ or
\[\la\in\{(2,1),(3,2),(3,2,1),(4,3,2),(4,3,2,1),(5,4,3,2),(5,4,3,2,1)\}.\]

By \cite[Theorem 1.1]{fm} $\bt\la$ is homogeneous in characteristic $3$ only if one of the following holds:
\begin{enumerate}
\item
$\la=(n)$;
\item
$\la_1\equiv\dots\equiv\la_{h(\la)}\equiv a\ppmod{3}$ with $a\in\{1,2\}$ and $h(\la)\geq 2$;
\item
$\la=(3k+a,3k+a-3,\dots,a)\sqcup(3)$ with $a\in\{1,2\}$ and $k\geq0$;
\item
$\la$ is one of the partitions $(2,1)$, $(3,2,1)$, $(4,3,2)$, $(4,3,2,1)$, $(5,3,2,1)$, $(5,4,3,1)$, $(5,4,3,2)$, $(5,4,3,2,1)$, $(7,4,3,2,1)$, $(8,5,3,2,1)$.
\end{enumerate}
So we just need to show that the \lcnamecref{mainhom} holds in each of these four cases.
\begin{enumerate}
\item
If $\la=(n)$, then $\bt\la$ is homogeneous in every characteristic by \cite[Table III]{wal}.
\item
Suppose $\la_1\equiv\dots\equiv\la_{h(\la)}\equiv a\ppmod{3}$ with $a\in\{1,2\}$ and $h(\la)\geq 2$. Then in particular $\la_i-\la_{i+1}\geq 3$ for every $1\leq i<h(\la)$, so there are addable and removable nodes in every row of $\la$.

If there exists $1\leq i<h(\la)$ with $\la_i\not\equiv\la_{i+1}\ppmod{2}$, then $\la_i+\la_{i+1}>1$ is odd and not divisible by $3$. Further $\la_i-\la_{i+1}$ is divisible by $3$. So we can apply \cref{addrem} with $r=i$, $s=i+1$ and $p$ any prime dividing $\la_i+\la_{i+1}$.

If $\la_i$ is even for all $1\leq i\leq h(\la)$ and $\la_1\equiv\la_2\ppmod{4}$, then we can apply \cref{addrem2} with $r=1$ and $s=2$ (note that $\la_1-\la_2$ is divisible by $3$, so cannot equal $4$).

If $\la_i$ is even for all $1\leq i\leq h(\la)$ and $\la_1\not\equiv\la_2\ppmod{4}$ then $\la_1+\la_2=2c$ with $c>1$ odd. As $\la_1-\la_2$ is also even, we can apply \cref{addrem} with $r=1$, $s=2$ and $p$ any prime dividing $c$.

We are left with the case where $\la_i$ is odd for all $1\leq i\leq h(\la)$. In this case, \cite[Theorem 5.1]{bo} shows that the $2$-modular reduction of $\bt\la$ has a composition factor appearing with multiplicity $1$. So $\bt\la$ is homogeneous in characteristic $2$ if and only if $\spe\la$ is irreducible in characteristic $2$. By \cite[Theorem 3.3]{mfspin2} it then in particular follows that $\la_1\equiv\la_2\ppmod{4}$. So again $\la_1+\la_2=2c$ with $c>1$ odd and we can conclude as in the previous case.
\item
Suppose $\la=(3k+a,3k+a-3,\dots,a)\sqcup(3)$ with $k\geq0$ and $a\in\{1,2\}$. If $a+k\ls3$ then we can just check the known decomposition numbers \cite{my,GAP,maas}, together with the fact that $\bt\la$ is automatically homogeneous in characteristic $p$ when $p>|\la|$. If $a=1$ and $k\gs3$, then we can apply \cref{addrem} with $r=k-2$, $s=k-1$ and $p=17$. If $a=2$ and $k\gs2$ then we can apply \cref{addrem} with $r=k-1$, $s=k$ and $p=13$.
\item
Suppose $\la$ is one of $(2,1)$, $(3,2,1)$, $(4,3,2)$, $(4,3,2,1)$, $(5,3,2,1)$, $(5,4,3,1)$, $(5,4,3,2)$, $(5,4,3,2,1)$, $(7,4,3,2,1)$, $(8,5,3,2,1)$. In all but the last case we can just check the known decomposition numbers \cite{my,GAP,maas}. In the last case we can apply \cref{addrem} with $r=1$ and $s=2$.
\end{enumerate}

We will now prove the last statement of the theorem. Assume that $G$ is a group containing $\haaa n$ as normal subgroup, that $N$ is a GIR of $G$ and that $M$ appears in $N\da^G_{\haaa n}$. Then by \cite[Proposition 2.8]{t2}, for any prime $p$, composition factors of the reduction modulo $p$ of $M$ are conjugate to each other under $\ol{\bbf_p}$ and $G$.

Consider first the case $n\not=6$. Then $\Aut(\aaa n)\leq \sss n$ and so automorphisms of $\haaa n$ are given by conjugation with elements of $\hsss n$. To see this let $\phi\in\Aut(\haaa n)$ and $\overline{\phi}\in\Aut(\aaa n)$ be the corresponding automorphism. Then $\overline{\phi}$ corresponds to conjugating by some element $g$ of $\sss n$. Let $\tilde{g}\in\hsss n$ be a lift of $g$. Then for every $h\in\haaa n$ there exists $c_h\in\{0,1\}$ with $\phi(h)=z^{c_h}\tilde{g}h\tilde{g}^{-1}$. Comparing order of elements we have that $c_h=0$ if $h$ is the lift of an element with odd order (in this case one of the two lifts has odd order while the other even order). Since elements of odd order generate $\haaa n$, we obtain $c_h=0$ for every $h$.

It follows that the reduction modulo $p$ of $M$ is almost homogeneous by \cref{spinirred,galoisspin} for $p>2$ and by \cref{anrep,galoisan} for $p=2$. In particular $M$ reduces almost homogeneously in every characteristic. If $n=6$ this last statement can be obtained by analysis of decomposition matrices.
\end{proof}

\begin{proof}[Proof of \cref{mains,maina}]
We may assume that $\la$ is one of the partitions appearing in \cref{mainhom}. For the seven sporadic partitions in \cref{mainhom}(2), we can just check the known decomposition numbers \cite{my,maas} to verify the result. 

This leaves the partition $\la=(n)$, for which $\spes\la$ is the so-called \emph{basic spin} module. The reducibility of a $p$-modular reduction of the basic spin module was determined completely by Wales \cite[Theorem 7.7]{wal}: $\spes\la$ is irreducible in characteristic $p$ \iff $n$ is even or $p\nmid n$. So if $n$ is even or $n=1$, then $\spes\la$ is irreducible in every characteristic. If $n\gs3$ is odd, then $\spe\la$ is reducible modulo any prime factor of $n$. So we have the desired result for $\spes\la$.

For $\apes\la$ we have a little more work to do. From \cite[Theorem 4.3]{mfspin2alt} we see that $\apes\la$ is irreducible in characteristic $2$ \iff $n=0$ or $n\nequiv0\ppmod 4$. To examine $\apes\la$ in odd characteristic, we note that in odd characteristic $p$, Wales's results (together with an analysis of when the partition $(n)$ is $p$-even) can be stated as saying that $\bt\la$ is an irreducible supermodule (isomorphic to $\bd\mu$, say) in characteristic $p$. As a consequence, if $p$ is odd, then $\apes\la$ is reducible in characteristic $p$ \iff $\la$ is odd and $\mu$ is $p$-even. Obviously $\la$ is odd \iff $n$ is even. On the other hand, the block classification for the double covers of symmetric and alternating groups in terms of residues \cite[Theorem 22.3.1(iii)]{kleshbook} shows that $\mu$ is $p$-even \iff $\la$ is, and it is easy to check that if $n$ is even, then $\la$ is $p$-even \iff $p\mid n$. We conclude that if $p$ is odd, then $\ape\la_\pm$ is irreducible in characteristic $p$ \iff $n$ is odd or $p\nmid n$.

So if $n$ is odd or if $n\ls2$, then $\ape\la_\pm$ is irreducible in every characteristic. If $n$ is divisible by $4$ and $n>0$, then $\ape\la$ is reducible in characteristic $2$. If $n\equiv2\ppmod4$ and $n>2$, then $\ape\la$ is reducible modulo $p$, where $p$ is any odd prime factor of $n$.
\end{proof}

\section{GIRs for double covers}\label{spingirsec}

Now we study GIRs for the double covers of $\sss n$ and $\aaa n$. Here it will be important to distinguish the two double covers $\hsss n^+$ and $\hsss n^-$. Given a strict partition $\la$ of $n$, we write $\spex\la\ep$ for the representation $\spe\la$ considered as a $\hsss n^\ep$-representation.

We will be concerned with the representations $\spes\la$ and $\apes\la$ that are almost homogeneous in every characteristic; that is, those appearing in \cref{mainhom}. The case $\la=(n)$ for $n\gs7$ is addressed in \cite{t1}, so we just need to look at the partition $(n)$ for $n\ls6$, together with the seven partitions in \cref{mainhom}(2). In Table I we list some essential information on the characters labelled by these partitions, including their Frobenius--Schur indicators, character fields and reductions modulo $p$. For a strict partition $\la$, we write $\chr\la$ or $\chr\la_\pm$ for the ordinary character of $\spe\la$ or $\spe\la_\pm$. In prime characteristic $p$, we write $\bchr\mu$ or $\bchr\mu_\pm$ for the Brauer characters of the appropriate simple modules labelled by $\mu$ (that is, the modules $\snd^\mu$, or $\ane^\mu_\ast$ if $p=2$, or the modules $\smp\mu_\ast$ or $\amp\mu_\ast$ if $p$ is odd).

\begin{figure}[p]
\begin{turn}{90}
\begin{minipage}{9in}
\scriptsize
\[
\begin{array}{ccccccccc}
\hline
\chi&\text{group}&\chi(1)&\bbq(\chi)&\ind(\chi)&p=2&p=3&p=5&p=7
\\\hline
\chr{\varnothing}&\hsss 0^+&1&\bbq&1&\bchr{\varnothing}&&&
\\
\chr{\varnothing}&\hsss 0^-&1&\bbq&1&\bchr{\varnothing}&&&
\\
\chr{\varnothing}&\haaa 0&1&\bbq&1&\bchr{\varnothing}&&&
\\
\hline
\chr{1}&\hsss 1^+&1&\bbq&1&\bchr{1}&&&
\\
\chr{1}&\hsss 1^-&1&\bbq&1&\bchr{1}&&&
\\
\chr{1}&\haaa 1&1&\bbq&1&\bchr{1}&&&
\\
\hline
\chr{2}_\pm&\hsss 2^+&1&\bbq(\sqrt{-1})&0&\bchr{2}&&&
\\
\chr{2}_\pm&\hsss 2^-&1&\bbq&1&\bchr{2}&&&
\\
\chr{2}&\haaa 2&1&\bbq&1&\bchr{2}&&&
\\
\hline
\chr{3}&\hsss 3^+&2&\bbq&-1&\bchr{2,1}&\bchr{2,1}_++\bchr{2,1}_-&&
\\
\chr{3}&\hsss 3^-&2&\bbq&1&\bchr{2,1}&\bchr{2,1}_++\bchr{2,1}_-&&
\\
\chr{3}_\pm&\haaa 3&1&\bbq(\sqrt{-3})&0&\bchr{2,1}_\pm&\bchr{2,1}&&
\\
\hline
\chr{4}_\pm&\hsss 4^+&2&\bbq(\sqrt{2})&-1&\bchr{3,1}&\bchr{3,1}_\pm&&
\\
\chr{4}_\pm&\hsss 4^-&2&\bbq(\sqrt{-2})&0&\bchr{3,1}&\bchr{3,1}_\pm&&
\\
\chr{4}&\haaa 4&2&\bbq&-1&\bchr{3,1}_++\bchr{3,1}_-&\bchr{3,1}&&
\\
\hline
\chr{5}&\hsss 5^+&4&\bbq&-1&\bchr{3,2}&\bchr{3,2}&\bchr{4,1}_++\bchr{4,1}_-&
\\
\chr{5}&\hsss 5^-&4&\bbq&-1&\bchr{3,2}&\bchr{3,2}&\bchr{4,1}_++\bchr{4,1}_-&
\\
\chr{5}_\pm&\haaa 5&2&\bbq(\sqrt{5})&-1&\bchr{3,2}_\pm&\bchr{3,2}_\pm&\bchr{4,1}&
\\
\hline
\chr{6}_\pm&\hsss 6^+&4&\bbq(\sqrt{-3})&0&\bchr{6,4}&\bchr{3,2,1}&\bchr{5,1}_\pm&
\\
\chr{6}_\pm&\hsss 6^-&4&\bbq(\sqrt{3})&-1&\bchr{6,4}&\bchr{3,2,1}&\bchr{5,1}_\pm&
\\
\chr{6}&\haaa 6&4&\bbq&-1&\bchr{6,4}&\bchr{3,2,1}_++\bchr{3,2,1}_-&\bchr{5,1}&
\\
\hline
\chr{2,1}_\pm&\hsss 3^+&1&\bbq(\sqrt{-1})&0&\bchr{3}&\bchr{2,1}_\pm&&
\\
\chr{2,1}_\pm&\hsss 3^-&1&\bbq&1&\bchr{3}&\bchr{2,1}_\pm&&
\\
\chr{2,1}&\haaa 3&1&\bbq&1&\bchr{3}&\bchr{2,1}&&
\\\hline
\chr{3,2}_\pm&\hsss 5^+&4&\bbq(\sqrt3)&-1&\bchr{4,1}&\bchr{3,2}&\bchr{3,2}_\pm&
\\
\chr{3,2}_\pm&\hsss 5^-&4&\bbq(\sqrt{-3})&0&\bchr{4,1}&\bchr{3,2}&\bchr{3,2}_\pm&
\\
\chr{3,2}&\haaa 5&4&\bbq&-1&\bchr{4,1}&\bchr{3,2}_++\bchr{3,2}_-&\bchr{3,2}&
\\
\hline
\chr{3,2,1}_\pm&\hsss 6^+&4&\bbq(\sqrt3)&-1&\bchr{5,1}&\bchr{3,2,1}&\bchr{3,2,1}_\pm&
\\
\chr{3,2,1}_\pm&\hsss 6^-&4&\bbq(\sqrt{-3})&0&\bchr{5,1}&\bchr{3,2,1}&\bchr{3,2,1}_\pm&
\\
\chr{3,2,1}&\haaa 6&4&\bbq&-1&\bchr{5,1}&\bchr{3,2,1}_++\bchr{3,2,1}_-&\bchr{3,2,1}&
\\
\hline
\chr{4,3,2}&\hsss 9^+&96&\bbq&1&2\bchr{6,3}&\bchr{4,3,2}_++\bchr{4,3,2}_-&\bchr{4,3,2}&\bchr{4,3,2}
\\
\chr{4,3,2}&\hsss 9^-&96&\bbq&-1&2\bchr{6,3}&\bchr{4,3,2}_++\bchr{4,3,2}_-&\bchr{4,3,2}&\bchr{4,3,2}
\\
\chr{4,3,2}_\pm&\haaa 9&48&\bbq(\sqrt{-6})&0&\bchr{6,3}&\bchr{4,3,2}&\bchr{4,3,2}_\pm&\bchr{4,3,2}_\pm
\\
\hline
\chr{4,3,2,1}&\hsss {10}^+&96&\bbq&1&2\bchr{7,3}&\bchr{4,3,2,1}_++\bchr{4,3,2,1}_-&\bchr{4,3,2,1}&\bchr{4,3,2,1}
\\
\chr{4,3,2,1}&\hsss {10}^-&96&\bbq&-1&2\bchr{7,3}&\bchr{4,3,2,1}_++\bchr{4,3,2,1}_-&\bchr{4,3,2,1}&\bchr{4,3,2,1}
\\
\chr{4,3,2,1}_\pm&\haaa {10}&48&\bbq(\sqrt{-6})&0&\bchr{7,3}&\bchr{4,3,2,1}&\bchr{4,3,2,1}_\pm&\bchr{4,3,2,1}_\pm
\\
\hline
\chr{5,4,3,2}&\hsss {14}^+&9152&\bbq&-1&2\bchr{8,5,1}&\bchr{5,4,3,2}_++\bchr{5,4,3,2}_-&\bchr{5,4,3,2}_++\bchr{5,4,3,2}_-&\bchr{5,4,3,2}
\\
\chr{5,4,3,2}&\hsss {14}^-&9152&\bbq&1&2\bchr{8,5,1}&\bchr{5,4,3,2}_++\bchr{5,4,3,2}_-&\bchr{5,4,3,2}_++\bchr{5,4,3,2}_-&\bchr{5,4,3,2}
\\
\chr{5,4,3,2}_\pm&\haaa {14}&4576&\bbq(\sqrt{-30})&0&\bchr{8,5,1}&\bchr{5,4,3,2}&\bchr{5,4,3,2}&\bchr{5,4,3,2}_\pm
\\
\hline
\chr{5,4,3,2,1}&\hsss {15}^+&9152&\bbq&-1&2\bchr{9,5,1}&\bchr{5,4,3,2,1}_++\bchr{5,4,3,2,1}_-&\bchr{5,4,3,2,1}_++\bchr{5,4,3,2,1}_-&\bchr{5,4,3,2,1}
\\
\chr{5,4,3,2,1}&\hsss {15}^-&9152&\bbq&1&2\bchr{9,5,1}&\bchr{5,4,3,2,1}_++\bchr{5,4,3,2,1}_-&\bchr{5,4,3,2,1}_++\bchr{5,4,3,2,1}_-&\bchr{5,4,3,2,1}
\\
\chr{5,4,3,2,1}_\pm&\haaa {15}&4576&\bbq(\sqrt{-30})&0&\bchr{9,5,1}&\bchr{5,4,3,2,1}&\bchr{5,4,3,2,1}&\bchr{5,4,3,2,1}_\pm
\\
\hline
\end{array}
\]
\begin{center}
Table I: list of small cases. We write $\chr\la,\chr\la_\pm$ for the ordinary irreducible characters of $\hsss n$ and $\haaa n$ labelled by $\la$, and $\bchr\la,\bchr\la_\pm$ for irreducible\\ Brauer characters in characteristic $p$. We only give reductions modulo $p$ in cases where the corresponding block has positive defect.
\end{center}
\end{minipage}
\end{turn}
\end{figure}

In fact, most of this section will be devoted to studying the modules $\spe{4,3,2}$, $\spe{4,3,2,1}$, $\spe{5,4,3,2}$ and $\spe{5,4,3,2,1}$, which are difficult to deal with. In particular we need to determine whether they are defined over $\bbq$ or $\bbq_p$ as $\hsss n^\ep$-representations for specific $p$ and $\ep$.

\subsection{Quaternion algebras}

We will need several results on quaternion algebras. We use the standard notation $\qalg{a,b}\bbf$ for the quaternion algebra over a field $\bbf$ with parameters $a,b\in\bbf$; that is, the $\bbf$-algebra generated by two elements $i$ and $j$ with defining relations $i^2=a$, $j^2=b$, $ji=-ij$.

We begin with the following result, which studies the structure of certain quaternion algebras over $\bbq$. We give a proof of it, as we are unaware of any previous proof. For $d\in\bbq$ with $\sqrt d\notin\bbq$ we write $m\mapsto\ol m$ for the non-trivial automorphism of the field $\bbq(\sqrt d)$. As usual, $\Mat_m(A)$ denotes the algebra of $m\times m$ matrices over a commutative algebra $A$.

\begin{lemma}\label{divalgebra}
Suppose $D$ is a $\bbq$-subalgebra of $\Mat_2(\bbq(\sqrt d))$ and is a quaternion division $\bbq$-algebra.
\begin{enumerate}
\item
There exist $f\in\bbq(\sqrt d)$, $m\in\bbq(\sqrt d)^\times$ and $h\in\bbq^\times$ such that $\spmx{\sqrt d}f0{-\sqrt d},\spmx0m{h\ol m}0\in D$.
\item
If $f=0$, then there is $k\in\bbq^\times$ such that
\[
D=\left\langle\pmx1001,\pmx{\sqrt{d}}00{-\sqrt{d}},\pmx01k0,\pmx0{\sqrt{d}}{-k\sqrt{d}}0\right\rangle_{\Q},
\]
so that $D\cong\qalg{d,k}\bbq$.
\end{enumerate}
\end{lemma}

\begin{proof}
Throughout this proof we write $\lan a_1,\dots,a_r\ran$ for the $\bbq$-span of $a_1,\dots,a_r\in\Mat_2(\bbq(\sqrt d))$. By assumption we can write $D=\lan I,A,B,AB\ran$, where $A^2=aI$, $B^2=bI$, and $AB=-BA$, and $a,b\in\bbq$. Then $(AB)^2=-abI$, and the matrices $A$, $B$, $AB$ pairwise anti-commute.

\begin{description}[leftmargin=\parindent]
\item[Claim 1] If $C\in\lan A,B,AB\ran$, then $C^2\in\lan I\ran$.

To see this, write $C=rA+sB+tAB$ with $r,s,t\in\bbq$. Then the assumptions on $A$ and $B$ give $C^2=(r^2a+s^2b-t^2ab)I$, with $r^2a+s^2b-t^2ab\in\bbq$.

%
\item[Claim 2] If $C\in\lan A,B,AB\ran$, then $C$ has trace $0$.

To see this, write $C=\spmx{c_1}{c_2}{c_3}{c_4}$. Claim 1 implies in particular that $(c_1+c_4)c_2=(c_1+c_4)c_3=0$, so that either $c_1+c_4=0$ or $c_2=c_3=0$. But in the latter case the diagonal entries of $C^2$ are $c_1^2$ and $c_4^2$, so we get $c_1=\pm c_4$. So either $c_1=-c_4$ (as required) or $C=c_1I$ for some $c_1\in\bbq(\sqrt d)^\times$. Now $c_1$ cannot be rational, because $I,A,B,AB$ are linearly independent over $\bbq$. But if $c_1$ is irrational, then $D$ contains $\bbq(\sqrt d)I$, so is a $\bbq(\sqrt d)$-subalgebra of $\Mat_2(\bbq(\sqrt d))$. Then $\dim_{\bbq(\sqrt d)}D=2$, which forces $D$ to be commutative, a contradiction.

\item[Claim 3] There exist $m\in\bbq(\sqrt d)^\times$ and $h\in\bbq^\times$ such that $\spmx0m{h\ol m}0\in D$.

To see this, note that the upper-left entries of the matrices $A$, $B$ and $AB$ are linearly dependent over $\bbq$ (because they lie in $\bbq(\sqrt d)$). So we can find a non-trivial $\bbq$-linear combination $C$ of $A$, $B$ and $AB$ such that the upper-left entry of $C$ (and hence the lower-right entry, by Claim 2) is zero. The off-diagonal entries of $C$ are non-zero because $C$ is invertible, so we can certainly write $C=\spmx0m{h\ol m}0$ with $h,m\in\bbq(\sqrt d)^\times$. Now $C^2=hm\ol mI$, and $m\ol m\in\bbq$, so $h\in\bbq$ by Claim 1.

\item[Claim 4] There exists $f\in\bbq(\sqrt d)$ such that $\spmx{\sqrt d}f0{-\sqrt d}\in D$.

Using the fact that the lower-left entries of $A$, $B$ and $AB$ are linearly dependent over $\bbq$ we can find a non-zero upper-triangular matrix $E\in\lan A,B,AB\ran$. By Claim 2 we can write $E=\spmx ef0{-e}$ for $e,f\in\bbq(\sqrt d)$, and $e\neq0$ because $E$ is invertible. Now $E^2=e^2I$, so $e^2$ is rational by Claim 1, which means that either $e\in\bbq$ or $e\in\sqrt d\bbq$. But if $e\in\bbq$ then $D$ contains the non-invertible matrix $E+eI=\spmx{2e}f00$, a contradiction. So $e\in\sqrt d\bbq$, and by rescaling we may assume $e=\sqrt d$.
\end{description}
This completes the proof of part 1 of the \lcnamecref{divalgebra}. Now suppose $f=0$. Then
\[
D=\left\langle\pmx1001,\pmx{\sqrt{d}}00{-\sqrt{d}},\pmx0m{k\ol m}0,\pmx0{m\sqrt{d}}{-k\ol m\sqrt{d}}0\right\rangle.
\]
Because the two matrices $\spmx0m{k\ol m}0$ and $\spmx0{m\sqrt{d}}{-k\ol m\sqrt{d}}0$ are linearly independent over $\bbq$, the upper-right entries $m$ and $m\sqrt d$ of these matrices are linearly independent over $\bbq$ (otherwise we would be able to find a non-zero matrix of the form $\spmx00\ast0$ in $D$, contradicting the assumption that $D$ is a division algebra). Hence there is a $\bbq$-linear combination of $\spmx0m{k\ol m}0$, $\spmx0{m\sqrt{d}}{-k\ol m\sqrt{d}}0$ of the form $\spmx01k0$, and $k\in\bbq^\times$ because of Claim 1 and the assumption that $D$ is a division algebra. So
\[
D=\left\langle\spmx1001,\spmx{\sqrt{d}}00{-\sqrt{d}},\spmx01k0,\spmx0{\sqrt{d}}{-k\sqrt{d}}0\right\rangle.\qedhere
\]
\end{proof}

Now we collect a few lemmas on ramification of quaternion algebras. Recall that the quaternion algebra $\qalg{a,b}\bbq$ is \emph{ramified} at a prime $p$ if the algebra $\qalg{a,b}{\bbq_p}$ is a division algebra.

As with \cref{divalgebra}, we give a proof of the following lemma, although we do not know whether it is new.

\begin{lemma}\label{ram}\indent
\begin{enumerate}
\vspace*{-\topsep}
\item
Suppose $a,b\in\bbz$, with $a\equiv2\ppmod8$ and $b$ odd. Then the algebra $\qalg{a,b}\bbq$ is ramified at $p=2$ \iff $b\equiv\pm3\ppmod8$.
\item
The algebra $\qalg{-2,-15}\bbq$ is ramified at $p=5$.
\end{enumerate}
\end{lemma}

\begin{proof}
By \cite[Lemma 1.1.3]{gs} the algebra $\qalg{a,b}{\bbq_p}$ fails to be a division algebra \iff we can find a non-zero element $x+yi+zj+wk\in\qalg{a,b}{\bbq_p}$ with norm zero, i.e.\ a solution to the equation
\[
x^2-ay^2-bz^2+abw^2=0
\]
for $x,y,z,w\in\bbq_p$ not all zero.
\begin{enumerate}
\item
First suppose $b\equiv1\ppmod8$, and let $(y,z,w)=(0,1,0)$. Then the above equation becomes $x^2=b$. Clearly this equation has a solution for $x$ modulo $8$, and therefore (by \cite[Exercise 6 on p.19]{koblitz}) has a solution for $x\in\bbz_2$.

When $b\equiv-1\ppmod8$, let $(y,z,w)=(1,1,0)$. Then the above equation becomes $x^2=a+b$. Again, this has a solution modulo $8$, and therefore has a solution in $\bbz_2$.

Now suppose $b\equiv\pm3\ppmod8$, and suppose $x,y,z,w\in\bbq_2$ are not all zero. By rescaling, we can assume $x,y,z,w\in\bbz_2$, and that $x,y,z,w$ are not all divisible by $2$. We want to show that $x^2-ay^2-bz^2+abw^2\not=0$. By definition of $\bbz_2$, it is enough to show that $x^2-ay^2-bz^2+abw^2\nequiv 0\ppmod{2^h}$ for some $h\geq 1$. As $x,y,z,w\in\bbz_2$, we may also reduce each of them modulo $2^h$ before computing $x^2-ay^2-bz^2+abw^2\ppmod{2^h}$.

Now we can just check all possibilities for $x,y,z,w$ modulo $16$ to show that $x^2-ay^2-bz^2+abw^2\nequiv0\ppmod{16}$, and hence $x^2-ay^2-bz^2+abw^2\neq0$.
\item
We have to show that the equation
\[
x^2+2y^2+15z^2+30w^2=0
\]
has no non-trivial solution in $\bbq_5$. Assume $x,y,z,w\in\bbq_5$ are not all zero. By rescaling, we can assume $x,y,z,w\in\Z_5$ and at least one of them is not divisible by $5$. Similar to case (1), by checking all possibilities for $x,y,z,w$ modulo $25$, we can check that $x^2+2y^2+15z^2+30w^2\nequiv0\ppmod{25}$, and therefore $x^2+2y^2+15z^2+30w^2\neq0$.\qedhere
\end{enumerate}
\end{proof}

The final result of this section compares ramifications of $K$ and $R$ under specific conditions.

\begin{lemma}\label{ramifiedKR}
Assume that $G$ is a finite group, $V$ a $\Q G$-representation and let $p$ be a prime. If $V\otimes_\Q\R$ is irreducible, $K=\End_{\Q G}(V)$ is a quaternion algebra and $R$ is a maximal order in $K$, then $K$ is ramified at $p$ if and only if $R$ is ramified at $p$.
\end{lemma}

\begin{proof}
Note that by \cite[Lemma 10.4.3]{voight}, $R_p:=R\otimes_{\Z}\Z_p$ is a maximal order of $K_p:=K\otimes_{\Q}\Q_p$. Since $R$ is a $\Z$-lattice, we have $R/pR\cong R_p/pR_p$ as algebras (through the standard isomorphism $\Z/p\Z\cong\Z_p/p\Z_p$). Further $|R/pR|=p^4$.

Assume first that $K$ is ramified at $p$, that is $K_p$ is a division algebra. By the proof of \cite[Proposition 3.2]{gr}, $R_p$ has a unique maximal two-sided ideal $I_p$. Since $I_p$ consists of all non-invertible elements of $R_p$, $R_p/I_p$ is a (skew) field. In particular $|R_p/I_p|\leq p^2$ in view of the first paragraph of \cite[\S3]{gr}.

Because $pR_p\subseteq I_p$, the uniqueness of $I_p$ as maximal ideal and the isomorphism $R/pR\cong R_p/pR_p$, there exists a unique maximal ideal $I$ of $R$ with $pR\subseteq I$. As $|R/I|<|R/pR|$, we have by uniqueness of $I$ and the first paragraph of \cite[\S3]{gr} that $R$ is ramified at $p$.

Assume now that $K$ is unramified at $p$, that is $K_p\cong\Mat_2(\Q_p)$. Then $R/pR\cong R_p/pR_p\cong \Mat_2(\F_p)$ by the proof of \cite[Proposition 3.2]{gr}. So $pR$ is a maximal two-sided ideal of $R$ and then $R$ is unramified at $p$.
\end{proof}

\subsection{The cases $(4,3,2)$ and $(4,3,2,1)$}

Now we look at our first two difficult cases.

\begin{lemma}\label{defoverQ}
Suppose that $\la=(4,3,2)$ or $(4,3,2,1)$. Then $\spex\la+$ is defined over $\bbq$.
\end{lemma}

\begin{proof}
Table I shows that the character of $\spex\la+$ is real-valued, so by \cite[Corollary 2.4]{feit} the Schur index of $\spex\la+$ over $\bbq$ is either $1$ or $2$. So there certainly exists an irreducible $\Q\hsss n^+$-representation $V$ with $V\otimes_\Q\C\cong\spex\la+^{\oplus 2}$. Now consider the algebra $\End_{\Q\hsss n^+}(V)$. Since $\End_{\Q\hsss n^+}(V)\otimes_\bbq\bbc=\End_{\bbc\hsss n^+}(V\otimes_\bbq\bbc)\cong\Mat_2(\bbc)$, the algebra $\End_{\Q\hsss n^+}(V)$ is a $4$-dimensional central $\bbq$-algebra. So by \cite[Proposition 1.2.1]{gs}, $\End_{\Q\hsss n^+}(\spex\la+)$ is a quaternion algebra.

In view of \cite[Theorem 2.14]{feit}, in order to show that $\spex\la+$ is defined over $\bbq$ it is enough to show that it is defined over $\bbr$ and over $\bbq_p$ for every prime $p$. For $\R$ this holds by \cite[Theorem 2.7]{feit}. Furthermore, if $p$ is a prime for which $\spex\la+$ is absolutely irreducible modulo $p$, then $\spex\la+$ is defined over $\bbq_p$ by \cite[Theorem 2.10]{feit}. From Table I, this only leaves us to consider the primes $p=2$ and $3$. 

In fact for $p=3$ 
we can still use \cite[Theorem 2.10]{feit}, since for $p=3$ the character field of $\smp\la_\pm$ is $\F_3$.  To see this note that, looking at known decomposition matrices, we see that, in the Grothendieck group, $[\smp\la_\pm]=[V_\pm]-[W_\pm]$ with $V_\pm$ and $W_\pm$ spin representations in characteristic 0 of dimension $160$ and $112$ if $\la=(4,3,2)$, or $448$ and $400$ if $\la=(4,3,2,1)$. Using \cite{GAP} to compute the character table of $\hsss n^+$, we see that any entry in the character values of such modules $V_\pm$ or $W_\pm$ is either integer, $\pm\sqrt{10}$ or $\pm\sqrt{7}$.

So for $\la=(4,3,2)$ or $(4,3,2,1)$ the algebra $\End_{\bbq\hsss n^+}(V)$ is unramified at $0$ and at any odd prime. But by \cite[Corollary 14.2.3]{voight} any quaternion $\bbq$-algebra is ramified at an even number of places, and therefore is unramified at $2$ as well. So for $F=\bbr$ or $\bbq_p$ with $p$ any prime, the algebra $\End_{F\hsss n^+}(V)$ is not a division algebra, so is isomorphic to $\Mat_2(F)$, and therefore $\spex\la+$ is defined over $F$.
\end{proof}

\begin{lemma}\label{432notgir}
Suppose that $\la=(4,3,2)$ or $(4,3,2,1)$ and let $n=|\la|$. Then $\spex\la-$ does not appear in a GIR of $\hsss n^-$.
\end{lemma}

\begin{proof}
By \cref{defoverQ} $\spex\la+$ can be defined over $\bbq$. Let $\rho$ be such a matrix representation. Then  (following \cite[p. 93]{stem}) we obtain a matrix representation for $\spex\la-$ over $\bbq(i)$ via $s_{j,-}\mapsto i\rho(s_{j,+})$.

Mapping $A+iB\in\Mat_m(\bbq(i))$ with $A,B\in\Mat_m(\bbq)$ to $A\otimes\begin{pmatrix}1&0\\0&1\end{pmatrix}+B\otimes\begin{pmatrix}0&1\\-1&0\end{pmatrix}$, we obtain a representation $\ol\rho$ of $\hsss n^-$ over $\bbq$, such that extending scalars to $\bbc$ gives $\spe\la^{\oplus 2}$ (though $\ol\rho$ is irreducible over $\bbq$).

We can view $\haaa n$ as a subgroup of $\hsss n^+$ and of $\hsss n^-$. We will use the isomorphism between these two copies of $\haaa n$ given by
\[
g_{+}=s_{j_1,+}\cdots s_{j_{2h},+}\mapsto z^hs_{j_1,-}\cdots s_{j_{2h},-}=g_-.
\]
Under this isomorphism, we obtain $\overline{\rho}(g_-)=\begin{pmatrix}\rho(g_{+})&0\\0&\rho(g_{+})\end{pmatrix}$ for $g_-\in\haaa n$.

By \cref{gross} and Table I, $\ape\la_+\oplus\ape\la_-=\spe\la\da_{\haaa n}$ is a GIR for $\haaa n$. Since the modules $\apepm\la$ both have character field $\Q(\sqrt{-6})$, \cref{tiep1} gives $\End_{\Q\haaa n}(\rho)\cong \Q(\sqrt{-6})$. Since $\overline{\rho}(g_-)=\begin{pmatrix}\rho(g_{+})&0\\0&\rho(g_{+})\end{pmatrix}$ for $g_-\in\haaa n$, it follows that $\End_{\Q\haaa n}(\overline\rho)=\Mat_2(\End_{\Q\haaa n}(\rho))$.

Assume for a contradiction that $\spex\la-$ appears in a GIR $V$, and let $K=\End_{\bbq\hsss n^-}(V)$ as in \cref{gir}. Then $K$ is a definite quaternion algebra by \cref{tiep1}. If we fix an isomorphism $\End_{\Q\haaa n}(\rho)\cong \Q(\sqrt{-6})$, then under the resulting isomorphism $\Mat_2(\End_{\Q\haaa n}(\rho))\cong\Mat_2(\bbq(\sqrt {-6}))$, $K$ corresponds to a $\bbq$-subalgebra $D\subset\Mat_2(\bbq(\sqrt{-6}))$, with $D$ being a quaternion division algebra. By \cref{divalgebra}(1) $D$ contains matrices of the form $\spmx{\sqrt{-6}}f0{-\sqrt{-6}},\spmx0m{h\ol m}0$, where $f\in\bbq(\sqrt{-6})$, $m\in\bbq(\sqrt{-6})^\times$ and $h\in\bbq^\times$. Now the matrix in $K$ corresponding to $\spmx{\sqrt{-6}}f0{-\sqrt{-6}}$ commutes with $\ol\rho(g)$ for every $g\in\hsss n^-$. In particular, it commutes with $\overline{\rho}(s_{1,-})=\spmx0{\rho(s_{1,+})}{-\rho(s_{1,+})}0$, which forces $f=0$. Now we can apply \cref{divalgebra}(2) to get
\[
D=\left\langle\pmx1001,\pmx{\sqrt{-6}}00{-\sqrt{-6}},\pmx01k0,\pmx0{\sqrt{-6}}{-k\sqrt{-6}}0\right\rangle_{\Q}
\]
for some $k\in\bbq^\times$. Because the matrix in $K$ corresponding to $\spmx01k0$ commutes with $\overline{\rho}(s_{1,-})=\spmx0{\rho(s_{1,+})}{-\rho(s_{1,+})}0$, we deduce that $k=-1$. So $K$ is the quaternion algebra $\qalg{-6,-1}\bbq$. By \cref{ram}(1), $K$ is unramified at $p=2$. Then $R$ is unramified at $p=2$ in view of \cref{ramifiedKR} and Table I. So $\spe\la$ does not appear in a GIR of $\hsss n^-$, by \cref{tiep2} and Table I.
\end{proof}

\subsection{The cases $(5,4,3,2)$ and $(5,4,3,2,1)$}

Now we come to the modules $\spex\la+$ for $\la=(5,4,3,2)$ or $(5,4,3,2,1)$. In the following lemma, which is a fixed-characteristic version of \cref{gross}(3), we use the same notation as in \cref{gir}. In particular $K=\End_{\bbq G}(V)$, $R\subseteq K$ is a maximal order and $I\subset R$ is an arbitrary maximal ideal.

\begin{lemma}\label{irred}
Let $G$ be a finite group and $V$ be an irreducible $\Q G$-representation with $\End_{\Q G}(V)$ a quaternion division algebra. Let $W$ be an irreducible composition factor of $V\otimes_\Q\C$ and $\chi$ be the character of $W$. Assume that for some prime $p$ one of the following holds:
\begin{itemize}
\item $\chi\equiv\rho\ppmod{p}$ for some absolutely irreducible $p$-Brauer character $\rho$;

\item $\chi\equiv\rho+\rho^p\pmod{p}$ for some absolutely irreducible $p$-Brauer character $\rho$ with $\F_p(\rho)=\F_{p^2}$.
\end{itemize}
Then $\Lambda/I\Lambda$ is an irreducible $(R/I)G$-representation.
\end{lemma}

\begin{proof}
Note that in either case $\chi$ is irreducible as an $\F_p G$-character. Let $K:=\End_{\Q G}(V)$. By \cref{tiep1} we have $V\otimes_\Q\C\cong W^{\oplus 2}$.

Assume first that $R$ is ramified at $p$. Then $R/I=\F_{p^2}$, by the proof of \cite[Proposition 2.7]{t2}. Further if $\psi$ is the character of $\Lambda/I\Lambda$ as an $(R/I)G$-representation, then the character $\phi$ of $\Lambda/p\Lambda$ as an $\F_pG$-representation satisfies $\phi\equiv 2(\psi+\psi^p)\ppmod{p}$. Since $\psi=2\chi$ this gives $\chi\equiv \psi+\psi^p\ppmod{p}$. This means that we are in the second case in the lemma and $\psi\equiv\rho\text{ or }\rho^p\ppmod{p}$ is absolutely irreducible. In particular $\Lambda/I\Lambda$ is irreducible as as $(R/I)G$-representation.

Assume now that $R$ is unramified at $p$. By the proof of \cite[Proposition 2.7]{t2} $R/I=\Mat_2(\F_p)$ in this case. Further $\chi$ is irreducible as an $\F_pG$-character, and in the Grothendieck group of $\F_pG$-representations $[\Lambda/p\Lambda]=2[D]$ with $D$ irreducible. Let $W\subseteq\Lambda/p\Lambda$ with $W$ irreducible as an $(R/I)G$-representation. By \cite[Lemma 2.5]{t2}, $[W]=2[E]$ with $E$ irreducible as $\F_pG$-module. So $W=\Lambda/p\Lambda$ is an irreducible $(R/I)G$-representation. So the lemma follows, as $(p)\subseteq I$ and then $\Lambda/I\Lambda$ is a quotient of $\Lambda/(p)\Lambda$.
\end{proof}

\begin{lemma}\label{ramified25}
Suppose that $\la=(5,4,3,2)$ or $(5,4,3,2,1)$ and let $n=|\la|$. Let $V$ be a representation of $\bbq\hsss n^-$ with $V\otimes_\bbq\bbc\cong (\spex\la-)^{\oplus 2}$. Then $\End_{\bbq\hsss n^-}(V)$ is a quaternion algebra, which is ramified at $p=2$ and $5$ and unramified at all other places.
\end{lemma}

\begin{proof}
Let $K:=\End_{\bbq\hsss n^-}(V)$. Then $K$ is a quaternion algebra, as in the first paragraph of the proof of \cref{defoverQ}. By definition $K$ is unramified at a prime $p$ if and only if $K\otimes_\bbq\bbq_p$ is not a division algebra. This is then equivalent to $V\otimes_\bbq\bbq_p$ being reducible, which in turn is equivalent to $\spex\la+$ being defined over $\bbq_p\hsss n^-$. The same applies for $p=0$, with $\bbr$ in place of $\bbq_p$.

From \cite[Theorems 2.7, 2.10]{feit} and Table I it then follows that $K$ can only be ramified at $p=2$ or $5$. Since $K$ is ramified at an even number of places, it is thus enough to show that it is ramified at $p=5$.

By definition $\hsss{6,5,3}^-\leq\hsss{14}^-$. Since $\hsss{14}^-\leq\hsss{15}^-$ we can also view $\hsss{6,5,3}^-$ as a subgroup of $\hsss{15}^-$. Recall the reduced Clifford products introduced in \cref{char0}. Stembridge's spin version of the Littlewood--Richardson rule \cite[Theorem 8.1]{stem} shows that $\spex\la-$ appears exactly once in $\spex{4,2}-\otimes\spex{5}-\otimes\spex{3}-\ua_{\hsss{6,5,3}^-}^{\hsss n^-}$. So by \cite[Theorem 2.1]{feit} to show that $\spex\la-$ is not defined over $\bbq_5$ it suffices to show that $\spex{4,2}-\otimes\spex{5}-\otimes\spex{3}-$ is not defined over $\bbq_5$ (both representations have integer-valued characters by \cite[Theorems 8.8 and 10.1]{hohum} and \cite[Proposition 4.2]{stem}).

Let $W$ be a $\bbq(i)\hsss{6,5,3}^-$-representation with $W\otimes_{\bbq(i)}\bbc\cong (\spex{4,2}-\otimes\spex{5}-\otimes\spex{3}-)^{\oplus 2}$ and let $H:=\End_{\bbq(i)\hsss{6,5,3}^-}(W)$. Then $H$ is a quaternion algebra over $\bbq(i)$ (with the same proof as $K$ over $\bbq$).

Note that $-1$ is a square modulo $5$ and thus also in $\bbq_5$ by \cite[Theorem 3]{koblitz}, so that $\Q(i)\subseteq\Q_5$. Let $\overline{W}:=W\otimes_{\Q(i)}\Q_5$ and $\overline{H}:=\End_{\Q_5\hsss{6,5,3}^-}(\overline{W})$. As $\overline{H}\cong H\otimes_{\Q(i)}\Q_5$, it is also a quaternion algebra.

We can give a direct construction of the module $W$. For $1\leq j\leq 5$ let $\rho(s_{j,-})$ be the matrices defined in \cref{S42}. Similarly for $j\in\{1,2,3,4,6,7\}$ let $\psi(s_{j,-})$ be the matrices defined in \cref{S53}. It can be checked through direct computation that $\pi(z)=-I$ and
\[
\pi(s_j,-)=\left\{\begin{array}{ll}
\rho(s_{j,-})\otimes I\otimes\pmx100{-1}&j\in\{1,2,3,4,5\},\\
I\otimes\psi(s_{j-6,-})\otimes\pmx0110&j\in\{7,8,9,10,12,13\}
\end{array}\right.
\]
satisfy the braid relations for $\hsss{6,5,3}^-$ and thus define a representation of $\bbq(i)\hsss{6,5,3}^-$. Comparing characters it follows that $\pi\otimes_{\bbq(i)}\bbc\cong (\spex{4,2}-\otimes\spex{5}-\otimes\spex{3}-)^{\oplus 2}$. Thus we may take $W=\pi$.

If $A$ and $B$ are the matrices in \cref{S42,S53} then it can be checked again by direct computation that the matrices
\[I\otimes I\otimes\pmx1001,\quad A\otimes I\otimes\pmx0110,\quad I\otimes B\otimes\pmx100{-1},\quad A\otimes B\otimes\pmx01{-1}0\]
commute with the images of all standard generators of $\hsss{6,5,3}^-$ under $\overline{\pi}$ (which coincide with their images under $\pi$) and are thus in $\End_{\bbq_5\hsss{6,5,3}^-}(\overline{\pi})$. As this endomorphism ring is 4-dimensional and the four matrices above are linearly independent, it follows that
\[
H=\End_{\bbq_5\hsss{6,5,3}^-}(\overline{\pi})=\left\langle I\otimes I\otimes\pmx1001,A\otimes I\otimes\pmx0110,I\otimes B\otimes\pmx100{-1},A\otimes B\otimes\pmx01{-1}0\right\rangle_{\Q_5}.
\]
Using the fact that $A^2=-2I$ and $B^2=-15I$ it follows that $H$ is isomorphic to the quaternion algebra $\qalg{-2,-15}{\bbq_5}$. By \cref{ram}(2), $H$ is ramified at $p=5$.
\end{proof}

Now we can prove our main result about the cases $\la=(5,4,3,2)$ and $(5,4,3,2,1)$.

\begin{lemma}\label{5432gir}
Suppose that $\la=(5,4,3,2)$ or $(5,4,3,2,1)$ and let $n=|\la|$. Then $\spex\la+$ appears in a GIR of $\hsss n^+$.
\end{lemma}

\begin{proof}
By \cref{gross} and Table I, $\ape\la_+\oplus\ape\la_-$ is a GIR for $\haaa n$, and in particular can be defined over $\Q$. So let $M$ be a $\bbq\haaa n$-module such that $M\otimes_\bbq\bbc\cong\ape\la_+\oplus\ape\la_-$, and let $\rho:\haaa n\to\gl(M)$ be the corresponding representation. We want to construct the induced module $M\ua^{\hsss n^\pm}_{\haaa n}$. Note that ${ }^{s_{1,+}}g={ }^{s_{1,-}}g$ for any $g\in\haaa n$; this can be seen using the identification $s_{j,\mp}=is_{j,\pm}\in\C\hsss n^\pm$ from \cite[p.92]{stem}. This means that we can unambiguously write ${}^{s_1,\pm}g$ as ${}^{s_1}g$ for $g\in\haaa n$.

Using the coset representatives $\{1,s_{1,\pm}\}$ we obtain matrix representations $\psi^\pm=\rho\ua^{\hsss n^\pm}_{\haaa n}$ defined over $\Q$, with
\[
\psi^\pm(g)=\pmx{\rho(g)}00{\rho({ }^{s_1}g)}\text{ for $g\in\haaa n$,}\qquad\psi^\pm(s_{1,\pm})=\pmx0I{\mp I}0.
\]
Viewed as a representation over $\bbc$, $\psi^\pm$ is isomorphic to the underlying representation of
\[
M\ua^{\hsss n^\pm}_{\haaa n}\otimes_\bbq\bbc\cong(M\otimes_\bbq\bbc)\ua^{\hsss n^\pm}_{\haaa n}\cong(\ape\la_+\oplus\ape\la_-)\ua^{\hsss n^\pm}_{\haaa n}\cong\spex\la\pm^{\oplus 2}.
\]



Since ${ }^{s_1}M\cong M$, there exists a matrix $C$, defined over $\Q$, such that $C(\rho({ }^{s_1}g))C^{-1}=\rho(g)$ for every $g\in\haaa n$. Now define another representation $\pi^\pm$ by
\[\pi^\pm(g):=\pmx I00C\psi^\pm(g)\pmx I00{C^{-1}}.\]
Then (as a representation over $\bbc$) $\pi^\pm$ is also isomorphic to the underlying representation of $\spex\la\pm^{\oplus 2}$. Furthermore,
\[
\pi^\pm(g)=\pmx{\rho(g)}00{\rho(g)}\text{ for $g\in\haaa n$,}\qquad
\pi^\pm(s_{1,\pm})=\pmx0{C^{-1}}{\mp C}0.
\]

Now let $K^\pm:=\End_{\Q\hsss n^\pm}(\pi^\pm)$. Then $K^\pm$ is a quaternion algebra (as at the beginning of the proof of \cref{ramified25}), and
\[
K^\pm\subseteq\End_{\Q\haaa n}(\pi^\pm\da^{\hsss n^\pm}_{\haaa n})=\Mat_2(\End_{\Q\haaa n}(\rho)).
\]
By \cref{tiep1} and Table I, $\End_{\Q\haaa n}(\rho)\cong\Q(\sqrt{-30})$. Let $D^\pm$ be the image of $K^\pm$ under the corresponding isomorphism $\Mat_2(\End_{\Q\haaa n}(\rho))\cong\Mat_2(\Q(\sqrt{-30}))$. Then $D^\pm$ is a quaternion algebra, with $D^\pm\subset \Mat_2(\Q(\sqrt{-30}))$. By \cref{divalgebra}(1) $D^\pm$ contains a matrix of the form $\spmx{\sqrt{-30}}{f^\pm}0{-\sqrt{-30}}$. Since the corresponding matrix in $K^\pm$ commutes with $\pi^\pm(s_{1,\pm})$, it follows that $f^\pm=0$. Then by \cref{divalgebra}(2),
\[
D^\pm=\left\langle\pmx1001,\pmx{\sqrt{-30}}00{-\sqrt{-30}},\pmx01{k^\pm}0,\pmx0{\sqrt{-30}}{-k^\pm\sqrt{-30}}0\right\rangle_{\Q}
\]
for some $k^\pm\in\Q^\times$. Note that $\spmx01{k^\pm}0$ corresponds to $\spmx0I{k^\pm I}0$ in $K^\pm$. As this matrix commutes with $\pi^\pm(s_{1,\pm})$, it follows that $k^\pm C^{-1}=\mp C$. In particular $k^-=-k^+$, so we will write $k:=k^+$, with $k^-=-k$. Then
\[
K^\pm\cong D^\pm\cong\qalg{-30,\pm k}\bbq.
\]
Furthermore, by repeatedly applying the isomorphism $\qalg{-30,\pm k}\bbq\cong\qalg{-30,\pm\frac{15}2k}\bbq$, we can assume $k$ is odd. Now \cref{ram}(1) shows that $\qalg{-30,k}\bbq$ is ramified at $p=2$ \iff $\qalg{-30,-k}\bbq$ is. We know from \cref{ramified25} that $K^-$ is ramified at $p=2$, and therefore $K^+$ is as well.

%
%

Since $\spex\la+$ has Frobenius--Schur indicator $-1$ by Table I, it is not defined over $\R$ by \cite[Theorem 2.7]{feit}. So $\pi^+$ remains irreducible on extension of scalars to $\bbr$. Now fix a prime $p$ and let $\Lambda$ and $I\subseteq R\subseteq K^+$ be as in \cref{gir} for $\pi^+$ and $p$. If $p$ is odd then $\Lambda/I\Lambda$ is irreducible as $(R/I)\hsss n^+$-representation by \cref{irred} and Table I.

So we may assume that $p=2$. By the above, $K^+$ is ramified at $p=2$ and $\spex\la+$ is not defined over $\R$. So $R^+$ is also ramified at $p=2$ by \cref{ramifiedKR}. Table I shows that $[\Lambda/2\Lambda]=4[\snd^\mu]$ in the Grothendieck group of $\bbf_2\hsss n^+$, with $\mu=(8,5,1)$ or $(9,5,1)$. By the arguments in the proof of \cite[Proposition 2.7]{t2} we then have $R/I\cong\F_4$ and $\Lambda/I\Lambda\cong \snd^\mu$ is irreducible as an $(R/I)\hsss n^+$-representation.
\end{proof}

\subsection{Proof of the main result for GIRs}

Finally we can complete the classification of GIRs for $\hsss n$ and $\haaa n$, and prove \cref{maingirsp,maingirsm,maingira}. 

\begin{proof}[Proof of \cref{maingirsp,maingirsm,maingira}]
Assume that $\spes\la$ or $\apes\la$ is a composition factor of a GIR with character $\chi$. Then by \cref{tiep2} all constituents of $\chi$, viewed as a $p$-Brauer character for any prime $p$, are conjugate under the Galois action of $\overline{\bbf_p}$. By \cref{galoisan,galoisspin} and \cite[Theorem 11.5]{JamesBook}, it follows that $\spes\la$ or $\apes\la$ is almost homogeneous in characteristic $p$. This applies for every $p$, so $\la$ is one of the partitions appearing in \cref{mainhom}.

If $\la=(n)$ with $n\geq 7$, then we can simply use \cite[Theorems 1.1, 1.1', 1.2]{t1}. (Recall that \cite{t1} uses the opposite sign convention to ours.)

So we are left with the cases where $\la=(n)$ for $0\leq n\leq 6$ or one $\la$ is of the partitions in case (2) of \cref{mainhom}, which are exactly the cases considered in Table I. 
Assume $M=\spes\la$ or $\apes\la$ is one of these modules, and let $\psi$ be the character of $M$. If $\bbq\not=\bbq(\psi)\subseteq\bbr$ then $M$ does not appear in a GIR by \cref{tiep1}. If $\bbq(\psi)$ is an imaginary quadratic field then Table I shows that $\psi$ is absolutely irreducible when reduced modulo any prime, so $M$ appears in a GIR by \cref{gross}.

This leaves the cases where $\bbq(\psi)=\bbq$. Consider first the cases where $\ind(\psi)=1$. If $\psi$ is absolutely irreducible modulo every prime then $M$ appears in a GIR by \cref{gross}. On the other hand if the $2$-modular reduction of $\psi$ has two isomorphic composition factors then $M$ does not appear in a GIR by \cref{tiep2}. The only remaining case is the module $\spex3-$. In this case the $3$-modular reduction of $\psi$ is $\phi(2,1)_++\phi(2,1)_-$, and the Brauer characters $\phi(2,1)_\pm$ are integer valued (since $\hsss3^-$ is just the direct product of $\sss 3$ and the group of order $2$). So the Brauer characters $\phi(2,1)_+$ and $\phi(2,1)_-$ cannot be conjugate under the action of $\overline{\bbf_3}$, so $M$ does not appear in a GIR by \cref{tiep2}.


Now consider cases where $\bbq(\psi)=\bbq$ and $\ind(\psi)=-1$. Unless $\la$ is one of $(4,3,2)$, $(4,3,2,1)$, $(5,4,3,2)$ or $(5,4,3,2,1)$, Table I shows that for any prime $p$ the $p$-modular reduction of $\psi$ is either irreducible or is a sum $\rho_1+\rho_2$ of two distinct irreducible Brauer characters. In the latter case, it is easily checked (since $\rho_1$ and $\rho_2$ have degree at most $2$) that $(\rho_1)^p\equiv\rho_2\ppmod p$ and $\bbf_p(\rho_1)=\bbf_p(\rho_2)=\bbf_{p^2}$. So $M$ is a constituent of a GIR, by \cref{gross}. The remaining four cases, where $M$ is one of $\spex{4,3,2}-$, $\spex{4,3,2,1}-$, $\spex{5,4,3,2}+$ or $\spex{5,4,3,2,1}+$, have been checked in \cref{432notgir,5432gir}.
\end{proof}

\appendix

\section{Matrices for $\spex{4,2}-$}\label{S42}

We give matrices for generators of $\hsss 6^-$ for a matrix representation $\rho$ of $\spex{4,2}-$ defined over the field $\bbq(i)$, together with a matrix $A$ which anticommutes with $\rho(g)$ for $g\in\hsss 6^-\setminus\haaa 6$ and commutes with $\rho(g)$ for $g\in\haaa 6$. As this is a spin representation $\rho(z)=-I$.

To enable the reader to compute with these matrices, we present them as GAP code which can be pasted into a GAP session. (The reader should invoke \texttt{i:=E(4);} in GAP to define $i=\sqrt{-1}$.)

\allowdisplaybreaks
\begin{align*}
&\rho(s_{1,-})=
\\
\texttt{\small[}&\texttt{\small[0,i,-i,0,0,0,0,0,0,0,0,0,0,0,0,0,0,0,0,0],}
\\
&\texttt{\small[0,0,0,-i,0,0,0,0,0,0,0,0,0,0,0,0,0,0,0,0],}
\\
&\texttt{\small[i,0,0,-i,0,0,0,0,0,0,0,0,0,0,0,0,0,0,0,0],}
\\
&\texttt{\small[0,i,0,0,0,0,0,0,0,0,0,0,0,0,0,0,0,0,0,0],}
\\
&\texttt{\small[0,0,0,0,1,-1,0,0,0,0,0,0,0,-1+i,1-i,0,0,i,-1-i,0],}
\\
&\texttt{\small[0,0,0,0,0,-1,0,0,0,0,0,0,-1,0,0,1-i,0,0,0,-1-i],}
\\
&\texttt{\small[0,0,0,0,0,0,-1,1,0,0,0,0,i,0,0,1-i,-1+i,0,0,-i],}
\\
&\texttt{\small[0,0,0,0,0,0,0,1,0,0,0,0,0,i,1,0,0,-1+i,0,0],}
\\
&\texttt{\small[0,0,0,0,0,0,0,0,1,-1,0,0,0,-1/2+i/2,1/2-i/2,0,0,0,0,0],}
\\
&\texttt{\small[0,0,0,0,0,0,0,0,0,-1,0,0,-1/2,0,0,1/2-i/2,0,0,0,0],}
\\
&\texttt{\small[0,0,0,0,0,0,0,0,0,0,-1,1,i/2,0,0,1/2-i/2,0,0,0,0],}
\\
&\texttt{\small[0,0,0,0,0,0,0,0,0,0,0,1,0,i/2,1/2,0,0,0,0,0],}
\\
&\texttt{\small[0,0,0,0,0,0,0,0,0,0,0,0,1,-1/2-i/2,-1/2+i/2,0,0,0,0,0],}
\\
&\texttt{\small[0,0,0,0,0,0,0,0,0,0,0,0,1/2,-1,0,-1/2+i/2,0,0,0,0],}
\\
&\texttt{\small[0,0,0,0,0,0,0,0,0,0,0,0,-i/2,0,-1,1/2+i/2,0,0,0,0],}
\\
&\texttt{\small[0,0,0,0,0,0,0,0,0,0,0,0,0,-i/2,-1/2,1,0,0,0,0],}
\\
&\texttt{\small[0,0,0,0,0,0,0,0,0,0,0,0,0,0,0,0,1,-1,0,0],}
\\
&\texttt{\small[0,0,0,0,0,0,0,0,0,0,0,0,0,0,0,0,0,-1,0,0],}
\\
&\texttt{\small[0,0,0,0,0,0,0,0,0,0,0,0,0,0,0,0,0,0,-1,1],}
\\
&\texttt{\small[0,0,0,0,0,0,0,0,0,0,0,0,0,0,0,0,0,0,0,1]]}
\\[9pt]
&\rho(s_{2,-})=
\\
\texttt{\small[}&\texttt{\small[0,i,-i,0,0,0,0,0,0,1-i,-1+i,0,0,0,0,0,0,0,0,0],}
\\
&\texttt{\small[0,0,0,-i,0,0,0,0,1,0,0,-1+i,0,0,0,0,0,0,0,0],}
\\
&\texttt{\small[i,0,0,-i,0,0,0,0,-i,0,0,-1+i,0,0,0,0,0,0,0,0],}
\\
&\texttt{\small[0,i,0,0,0,0,0,0,0,-i,-1,0,0,0,0,0,0,0,0,0],}
\\
&\texttt{\small[0,0,0,0,1,-1,0,0,0,0,0,0,0,0,0,0,0,0,0,0],}
\\
&\texttt{\small[0,0,0,0,0,-1,0,0,0,0,0,0,0,0,0,0,0,0,0,0],}
\\
&\texttt{\small[0,0,0,0,0,0,-1,1,0,0,0,0,0,0,0,0,0,0,0,0],}
\\
&\texttt{\small[0,0,0,0,0,0,0,1,0,0,0,0,0,0,0,0,0,0,0,0],}
\\
&\texttt{\small[0,0,0,0,0,0,0,0,0,1,-1,0,0,0,0,0,0,0,0,0],}
\\
&\texttt{\small[0,0,0,0,0,0,0,0,1,0,0,-1,0,0,0,0,0,0,0,0],}
\\
&\texttt{\small[0,0,0,0,0,0,0,0,0,0,0,-1,0,0,0,0,0,0,0,0],}
\\
&\texttt{\small[0,0,0,0,0,0,0,0,0,0,-1,0,0,0,0,0,0,0,0,0],}
\\
&\texttt{\small[0,0,0,0,0,0,0,0,0,0,0,0,0,i,-i,0,0,0,0,0],}
\\
&\texttt{\small[0,0,0,0,0,0,0,0,0,0,0,0,0,0,0,-i,0,0,0,0],}
\\
&\texttt{\small[0,0,0,0,0,0,0,0,0,0,0,0,i,0,0,-i,0,0,0,0],}
\\
&\texttt{\small[0,0,0,0,0,0,0,0,0,0,0,0,0,i,0,0,0,0,0,0],}
\\
&\texttt{\small[0,0,0,0,0,0,0,0,0,0,0,0,0,0,0,0,0,1,-1,0],}
\\
&\texttt{\small[0,0,0,0,0,0,0,0,0,0,0,0,0,0,0,0,1,0,0,-1],}
\\
&\texttt{\small[0,0,0,0,0,0,0,0,0,0,0,0,0,0,0,0,0,0,0,-1],}
\\
&\texttt{\small[0,0,0,0,0,0,0,0,0,0,0,0,0,0,0,0,0,0,-1,0]]}
\\[9pt]
&\rho(s_{3,-})=
\\
\texttt{\small[}&\texttt{\small[1,-1,0,0,0,0,0,0,0,0,0,0,0,0,0,0,0,0,0,0],}
\\
&\texttt{\small[0,-1,0,0,0,0,0,0,0,0,0,0,0,0,0,0,0,0,0,0],}
\\
&\texttt{\small[0,0,-1,1,0,0,0,0,0,0,0,0,0,0,0,0,0,0,0,0],}
\\
&\texttt{\small[0,0,0,1,0,0,0,0,0,0,0,0,0,0,0,0,0,0,0,0],}
\\
&\texttt{\small[0,0,0,0,0,1,-1,0,0,0,0,0,0,0,0,0,0,0,0,0],}
\\
&\texttt{\small[0,0,0,0,1,0,0,-1,0,0,0,0,0,0,0,0,0,0,0,0],}
\\
&\texttt{\small[0,0,0,0,0,0,0,-1,0,0,0,0,0,0,0,0,0,0,0,0],}
\\
&\texttt{\small[0,0,0,0,0,0,-1,0,0,0,0,0,0,0,0,0,0,0,0,0],}
\\
&\texttt{\small[1,-1-i,i,0,0,1/2-i/2,-1/2+i/2,0,0,i,-i,0,0,0,0,0,0,0,0,0],}
\\
&\texttt{\small[0,-1,0,i,1/2,0,0,-1/2+i/2,0,0,0,-i,0,0,0,0,0,0,0,0],}
\\
&\texttt{\small[-i,0,-1,1+i,-i/2,0,0,-1/2+i/2,i,0,0,-i,0,0,0,0,0,0,0,0],}
\\
&\texttt{\small[0,-i,0,1,0,-i/2,-1/2,0,0,i,0,0,0,0,0,0,0,0,0,0],}
\\
&\texttt{\small[0,0,0,0,0,-1/2+i/2,1/2-i/2,0,0,0,0,0,0,i,-i,0,0,i,-1-i,0],}
\\
&\texttt{\small[0,0,0,0,-1/2,0,0,1/2-i/2,0,0,0,0,0,0,0,-i,0,0,0,-1-i],}
\\
&\texttt{\small[0,0,0,0,i/2,0,0,1/2-i/2,0,0,0,0,i,0,0,-i,-1+i,0,0,-i],}
\\
&\texttt{\small[0,0,0,0,0,i/2,1/2,0,0,0,0,0,0,i,0,0,0,-1+i,0,0],}
\\
&\texttt{\small[0,0,0,0,0,0,0,0,0,0,0,0,0,0,0,0,0,0,1,0],}
\\
&\texttt{\small[0,0,0,0,0,0,0,0,0,0,0,0,0,0,0,0,0,0,0,1],}
\\
&\texttt{\small[0,0,0,0,0,0,0,0,0,0,0,0,0,0,0,0,1,0,0,0],}
\\
&\texttt{\small[0,0,0,0,0,0,0,0,0,0,0,0,0,0,0,0,0,1,0,0]]}
\\[9pt]
&\rho(s_{4,-})=
\\
\texttt{\small[}&\texttt{\small[0,1,-1,0,0,0,0,0,0,0,0,0,0,0,0,0,0,0,0,0],}
\\
&\texttt{\small[1,0,0,-1,0,0,0,0,0,0,0,0,0,0,0,0,0,0,0,0],}
\\
&\texttt{\small[0,0,0,-1,0,0,0,0,0,0,0,0,0,0,0,0,0,0,0,0],}
\\
&\texttt{\small[0,0,-1,0,0,0,0,0,0,0,0,0,0,0,0,0,0,0,0,0],}
\\
&\texttt{\small[0,0,0,0,0,0,1,0,0,0,0,0,0,0,0,0,0,0,0,0],}
\\
&\texttt{\small[0,0,0,0,0,0,0,1,0,0,0,0,0,0,0,0,0,0,0,0],}
\\
&\texttt{\small[0,0,0,0,1,0,0,0,0,0,0,0,0,0,0,0,0,0,0,0],}
\\
&\texttt{\small[0,0,0,0,0,1,0,0,0,0,0,0,0,0,0,0,0,0,0,0],}
\\
&\texttt{\small[0,0,0,0,0,0,0,0,0,i,-i,0,0,0,0,0,0,0,0,0],}
\\
&\texttt{\small[0,0,0,0,0,0,0,0,0,0,0,-i,0,0,0,0,0,0,0,0],}
\\
&\texttt{\small[0,0,0,0,0,0,0,0,i,0,0,-i,0,0,0,0,0,0,0,0],}
\\
&\texttt{\small[0,0,0,0,0,0,0,0,0,i,0,0,0,0,0,0,0,0,0,0],}
\\
&\texttt{\small[0,0,0,0,0,0,0,0,0,0,0,0,0,1,-1,0,0,0,0,0],}
\\
&\texttt{\small[0,0,0,0,0,0,0,0,0,0,0,0,1,0,0,-1,0,0,0,0],}
\\
&\texttt{\small[0,0,0,0,0,0,0,0,0,0,0,0,0,0,0,-1,0,0,0,0],}
\\
&\texttt{\small[0,0,0,0,0,0,0,0,0,0,0,0,0,0,-1,0,0,0,0,0],}
\\
&\texttt{\small[0,0,0,0,0,0,0,0,0,0,0,0,0,-1+i,1-i,0,0,i,-i,0],}
\\
&\texttt{\small[0,0,0,0,0,0,0,0,0,0,0,0,-1,0,0,1-i,0,0,0,-i],}
\\
&\texttt{\small[0,0,0,0,0,0,0,0,0,0,0,0,i,0,0,1-i,i,0,0,-i],}
\\
&\texttt{\small[0,0,0,0,0,0,0,0,0,0,0,0,0,i,1,0,0,i,0,0]]}
\\[9pt]
&\rho(s_{5,-})=
\\
\texttt{\small[}&\texttt{\small[0,0,1,0,0,0,0,0,0,0,0,0,0,0,0,0,0,0,0,0],}
\\
&\texttt{\small[0,0,0,1,0,0,0,0,0,0,0,0,0,0,0,0,0,0,0,0],}
\\
&\texttt{\small[1,0,0,0,0,0,0,0,0,0,0,0,0,0,0,0,0,0,0,0],}
\\
&\texttt{\small[0,1,0,0,0,0,0,0,0,0,0,0,0,0,0,0,0,0,0,0],}
\\
&\texttt{\small[-1,1+i,-i,0,0,0,1,0,0,1-i,-1+i,0,0,0,0,0,0,0,0,0],}
\\
&\texttt{\small[0,1,0,-i,0,0,0,1,1,0,0,-1+i,0,0,0,0,0,0,0,0],}
\\
&\texttt{\small[i,0,1,-1-i,1,0,0,0,-i,0,0,-1+i,0,0,0,0,0,0,0,0],}
\\
&\texttt{\small[0,i,0,-1,0,1,0,0,0,-i,-1,0,0,0,0,0,0,0,0,0],}
\\
&\texttt{\small[0,0,0,0,0,0,0,0,0,1/2-i/2,1/2+i/2,0,0,0,0,0,0,0,0,0],}
\\
&\texttt{\small[0,0,0,0,0,0,0,0,1/2,0,0,1/2+i/2,0,0,0,0,0,0,0,0],}
\\
&\texttt{\small[0,0,0,0,0,0,0,0,1-i/2,0,0,-1/2+i/2,0,0,0,0,0,0,0,0],}
\\
&\texttt{\small[0,0,0,0,0,0,0,0,0,1-i/2,-1/2,0,0,0,0,0,0,0,0,0],}
\\
&\texttt{\small[0,0,0,0,0,0,0,0,0,-1/2+i/2,1/2-i/2,0,0,0,1,0,0,0,0,0],}
\\
&\texttt{\small[0,0,0,0,0,0,0,0,-1/2,0,0,1/2-i/2,0,0,0,1,0,0,0,0],}
\\
&\texttt{\small[0,0,0,0,0,0,0,0,i/2,0,0,1/2-i/2,1,0,0,0,0,0,0,0],}
\\
&\texttt{\small[0,0,0,0,0,0,0,0,0,i/2,1/2,0,0,1,0,0,0,0,0,0],}
\\
&\texttt{\small[0,0,0,0,0,0,0,0,0,0,0,0,0,0,0,0,0,i,-i,0],}
\\
&\texttt{\small[0,0,0,0,0,0,0,0,0,0,0,0,0,0,0,0,0,0,0,-i],}
\\
&\texttt{\small[0,0,0,0,0,0,0,0,0,0,0,0,0,0,0,0,i,0,0,-i],}
\\
&\texttt{\small[0,0,0,0,0,0,0,0,0,0,0,0,0,0,0,0,0,i,0,0]]}
\\[9pt]
&A=
\\
\texttt{\small[}&\texttt{\small[1,-1,-1,0,0,0,0,0,0,0,0,0,0,0,0,0,0,0,0,0],}
\\
&\texttt{\small[2,-1,0,-1,0,0,0,0,0,0,0,0,0,0,0,0,0,0,0,0],}
\\
&\texttt{\small[1,0,-1,1,0,0,0,0,0,0,0,0,0,0,0,0,0,0,0,0],}
\\
&\texttt{\small[0,1,-2,1,0,0,0,0,0,0,0,0,0,0,0,0,0,0,0,0],}
\\
&\texttt{\small[0,0,0,0,1,-1,-1,0,0,0,0,0,0,0,0,0,0,0,0,0],}
\\
&\texttt{\small[0,0,0,0,2,-1,0,-1,0,0,0,0,0,0,0,0,0,0,0,0],}
\\
&\texttt{\small[0,0,0,0,1,0,-1,1,0,0,0,0,0,0,0,0,0,0,0,0],}
\\
&\texttt{\small[0,0,0,0,0,1,-2,1,0,0,0,0,0,0,0,0,0,0,0,0],}
\\
&\texttt{\small[0,0,0,0,0,0,0,0,1,-1,-1,0,0,0,0,0,0,0,0,0],}
\\
&\texttt{\small[0,0,0,0,0,0,0,0,2,-1,0,-1,0,0,0,0,0,0,0,0],}
\\
&\texttt{\small[0,0,0,0,0,0,0,0,1,0,-1,1,0,0,0,0,0,0,0,0],}
\\
&\texttt{\small[0,0,0,0,0,0,0,0,0,1,-2,1,0,0,0,0,0,0,0,0],}
\\
&\texttt{\small[0,0,0,0,0,0,0,0,0,0,0,0,1,-1,-1,0,0,0,0,0],}
\\
&\texttt{\small[0,0,0,0,0,0,0,0,0,0,0,0,2,-1,0,-1,0,0,0,0],}
\\
&\texttt{\small[0,0,0,0,0,0,0,0,0,0,0,0,1,0,-1,1,0,0,0,0],}
\\
&\texttt{\small[0,0,0,0,0,0,0,0,0,0,0,0,0,1,-2,1,0,0,0,0],}
\\
&\texttt{\small[0,0,0,0,0,0,0,0,0,0,0,0,0,0,0,0,1,-1,-1,0],}
\\
&\texttt{\small[0,0,0,0,0,0,0,0,0,0,0,0,0,0,0,0,2,-1,0,-1],}
\\
&\texttt{\small[0,0,0,0,0,0,0,0,0,0,0,0,0,0,0,0,1,0,-1,1],}
\\
&\texttt{\small[0,0,0,0,0,0,0,0,0,0,0,0,0,0,0,0,0,1,-2,1]]}
\end{align*}

\section{Matrices for $\spex{5}-\otimes\spex{3}-$}\label{S53}

We give matrices for generators of $\hsss{5,3}^-$ for a matrix representation $\psi$ of $\spex{5}-\otimes\spex{3}-$ defined over $\bbq(i)$, together with a matrix $B$ which anticommutes with $\psi(g)$ for $g\in\hsss{5,3}^-\setminus\haaa{5,3}$ and commutes with $\rho(g)$ for $g\in\haaa{5,3}$. As this is a spin representation, $\psi(z)=-I$.

%

\begin{align*}
&\psi(s_{1,-})=
\\
\texttt{\small[}&\texttt{\small[1,-1,0,0,0,0,0,0],[0,-1,0,0,0,0,0,0],}
\\
&\texttt{\small[0,0,-1,1,0,0,0,0],[0,0,0,1,0,0,0,0],}
\\
&\texttt{\small[0,0,0,0,-1,1,0,0],[0,0,0,0,0,1,0,0],}
\\
&\texttt{\small[0,0,0,0,0,0,1,-1],[0,0,0,0,0,0,0,-1]]}
\\[9pt]
&\psi(s_{2,-})=
\\
\texttt{\small[}&\texttt{\small[0,1,-1,0,0,0,0,0],[1,0,0,-1,0,0,0,0],}
\\
&\texttt{\small[0,0,0,-1,0,0,0,0],[0,0,-1,0,0,0,0,0],}
\\
&\texttt{\small[0,0,0,0,0,-1,1,0],[0,0,0,0,-1,0,0,1],}
\\
&\texttt{\small[0,0,0,0,0,0,0,1],[0,0,0,0,0,0,1,0]]}
\\[9pt]
&\psi(s_{3,-})=
\\
\texttt{\small[}&\texttt{\small[0,0,1,0,-1,0,0,0],[0,0,0,1,0,-1,0,0],}
\\
&\texttt{\small[1,0,0,0,0,0,-1,0],[0,1,0,0,0,0,0,-1],}
\\
&\texttt{\small[0,0,0,0,0,0,-1,0],[0,0,0,0,0,0,0,-1],}
\\
&\texttt{\small[0,0,0,0,-1,0,0,0],[0,0,0,0,0,-1,0,0]]}
\\[9pt]
&\psi(s_{4,-})=
\\
\texttt{\small[}&\texttt{\small[0,0,0,0,1,0,0,0],[0,0,0,0,0,1,0,0],}
\\
&\texttt{\small[0,0,0,0,0,0,1,0],[0,0,0,0,0,0,0,1],}
\\
&\texttt{\small[1,0,0,0,0,0,0,0],[0,1,0,0,0,0,0,0],}
\\
&\texttt{\small[0,0,1,0,0,0,0,0],[0,0,0,1,0,0,0,0]]}
\\[9pt]
&\psi(s_{6,-})=
\\
\texttt{\small[}&\texttt{\small[0,0,i,0,-i,0,0,0],[0,0,0,i,0,-i,0,0],}
\\
&\texttt{\small[0,0,0,0,0,0,-i,0],[0,0,0,0,0,0,0,-i],}
\\
&\texttt{\small[i,0,0,0,0,0,-i,0],[0,i,0,0,0,0,0,-i],}
\\
&\texttt{\small[0,0,i,0,0,0,0,0],[0,0,0,i,0,0,0,0]]}
\\[9pt]
&\psi(s_{7,-})=
\\
\texttt{\small[}&\texttt{\small[0,i,-i,0,0,0,0,0],[0,0,0,-i,0,0,0,0],}
\\
&\texttt{\small[i,0,0,-i,0,0,0,0],[0,i,0,0,0,0,0,0],}
\\
&\texttt{\small[0,0,0,0,0,-i,i,0],[0,0,0,0,0,0,0,i],}
\\
&\texttt{\small[0,0,0,0,-i,0,0,i],[0,0,0,0,0,-i,0,0]]}
\\[9pt]
&B=
\\
\texttt{\small[}&\texttt{\small[-3,2,2,0,2,0,0,0],[-6,3,0,2,0,2,0,0],}
\\
&\texttt{\small[-4,0,3,-2,0,0,2,0],[0,-4,6,-3,0,0,0,2],}
\\
&\texttt{\small[-2,0,0,0,3,-2,-2,0],[0,-2,0,0,6,-3,0,-2],}
\\
&\texttt{\small[0,0,-2,0,4,0,-3,2],[0,0,0,-2,0,4,-6,3]]}
\end{align*}

\end{document}